\documentclass[11pt]{article}
\usepackage[latin1]{inputenc}
\usepackage{latexsym}
\usepackage{amssymb}
\usepackage{hyperref}

\def\R{{\sl I\!\kern-.27em R}}
\def\N{{\sl I\!\kern-.27em N}}
\def\P{{\sl I\!\kern-.27em P}}
\def\max{\mathop{\rm max}\nolimits}

\def\q{ \quad}
\def\dis{ \displaystyle}
\def\fra{\dis\frac}
\def\gradH{\nabla\!_{\bf x}} 
\def\gradV{\nabla}   
\def\laplac{\Delta}  
\def\uu{{\bf u}^{m+1}}
\def\umedio{{\bf u}^{m+1/2}}
\def\ee{{\bf e}^{m+1}}
\def\emedio{{\bf e}^{m+1/2}}

\def\deltaee{\delta_t {\bf e}^{m+1}}
\def\deltaemedio{\delta_t {\bf e}^{m+1/2}}
\newtheorem{theorem}{Theorem}
\newtheorem{lemma}[theorem]{Lemma}
\newtheorem{corollary}[theorem]{Corollary}
\newtheorem{remark}{Remark}

\newcommand{\cqfd}{\mbox{}\nolinebreak\hfill\rule{2mm}{2mm}\medskip\par}

\newenvironment{proof}[1] {\par\noindent{\bf Proof. }{#1}}{\cqfd}
\begin{document}
%
\title{Convergence and error estimates of a viscosity-splitting finite-element scheme for
  the  Primitive Equations
\thanks{The authors have been partially supported by MINECO (Spain),  Grant  MTM2012--32325
 and the second author is also partially supported  by the research group
FQM-315 of  Junta de Andalucía.}}

\author{F.~Guill\'en-Gonz\'alez\footnote{Departamento de
Ecuaciones Diferenciales y An\'alisis Num\'erico and IMUS. Universidad de
Sevilla. Aptdo. 11690, 41080 Sevilla (Spain),
email: guillen@us.es, phone: ++ 34 954559907.},
  M.V.~Redondo-Neble\footnote{Departamento de Matem\'aticas. Universidad de
C\'adiz. C.A.S.E.M. Pol\'{\i}gono R\'{\i}o San Pedro S/N, $11510$
Puerto Real. C\'adiz (Spain), email: victoria.redondo@uca.es,
 phone: ++ 34 956016058.} }
%
%
\date{}
\maketitle
\begin{abstract}
The purpose of this paper is the numerical analysis of a first
order fractional-step time-scheme, using decomposition of the
viscosity, and ``inf-sup" stable finite element space-approximations  for the 
Primitive Equations of the Ocean. The aim of the paper is twofold.
Firstly, we prove that the scheme is unconditionally  stable and
convergent towards  weak solutions of the 
Primitive Equations. Secondly,  optimal error estimates for  velocity and pressure are provided of order $O(k+h^l)$ for $l=1$ or $l=2$ when  either    first  or  second order finite-element approximations are considered ($k$ and $h$ being the time step and
the mesh size, respectively). In both cases, these error estimates are obtained  under the same constraint $k\le h^2$.   
\end{abstract}



\section*{Introduction}
Assuming some simplifications (basically hydrostatic pressure and
  ``rigid lid'' hypothesis), the $3D$ Navier-Stokes equations derive
to the so-called ``Primitive Equations'' (or 
Hydrostatic Navier-Stokes equations). These equations arise a general
mathematical problem in the field of geophysical fluids
(\cite{CB,ltw,pe}). In particular, they describe the large-scale motions in the ocean \cite{ltw2}. The rigid lid hypothesis (no vertical displacements of the free surface of the ocean)  is usually assumed in Oceanography,  except in the case when fast surface waves are of interest \cite{CB}.

For simplicity, we take constant density, Cartesian
coordinates ($x$ in the easterly direction, $y$ in the northerly direction and
 $z$ perpendicular to the surface of the Earth) and we assume that the
effects due to temperature and salinity can be decoupled from the flow
dynamic. Then, the Primitive Equations model can be written as
 (\cite{l,ltw,ltw2}):
$$
 \left \{ \begin{array}{rl}
 \partial_t {\bf u} \,+\, ( {\bf U}\cdot \gradV) {\bf u}\,
 - \nu\laplac {\bf u} +\,{\bf b}({\bf u})\, +
               \gradH \, p \,
            =
   \, {\bf f} &  \hbox{ in $ \Omega \times (0,T)$,} \\
\noalign{\smallskip}
 \partial_z  p   =  -\rho \, g , \qquad
\gradV \cdot  {\bf U}\,  =  \, 0 &  \hbox{ in $ \Omega \times (0,T)$,}\\
        \noalign{\smallskip}
{\bf u}  = u_3 n_3 = 0 &  \hbox{ on $ \Gamma _b \times (0,T)$,}
\\ \noalign{\smallskip}
{\bf u}=0 & \hbox{ on $ \Gamma_l \times (0,T)$,}
 \\  \noalign{\smallskip}
 \nu \, \partial_z  {\bf u}\, = \, {\bf g}_s,\ \ \ u_3 = 0 &
\hbox{ on $ \Gamma _s \times (0,T)$,}\\
 {\bf u}_{\vert t=0} = {\bf u}_0  &   \mbox{ in } \Omega,
           \end{array} \right.
\leqno{(P)}
 $$
where $\Omega=\{ ({\bf x},z)\in \R ^3 : \ {\bf  x}=(x,y) \in S , \
-D({\bf x}) <z<0 \}$ is the  $3D$ domain  filled by the water, with
$S \subset \R ^2$  the surface domain (a regular bounded $2D$
domain) and $D:\overline{S} \rightarrow \R_+$
 (with $D>0$ in $S$)  the bottom function.
 Then, the different boundaries of  $\Omega$ are denoted as $\Gamma_s = \overline{S}\times \{0\}$   the surface, $\Gamma_b = \{({\bf x},-D({\bf x})):\ {\bf x}\in S\}$  the bottom  and $\Gamma_l= \{({\bf x},z):\ {\bf x}\in \partial
S , -D({\bf x}) <z<0\}$ the lateral walls (with outwards normal vector $({\bf
n}_{\bf x},n_3)$). Note that if $D\in C(\overline{S})$ then the bottom has no steps and the condition $u_3 n_3=0$ on $\Gamma _b \times (0,T)$ derives
 to $u_3=0$ on $\Gamma _b \times (0,T)$. In fact, $\Omega$ is the non-dimensional  domain obtained after a vertical scaling  and problem (P) appears as asymptotic limit from the anisotropic Navier-Stokes equations when the aspect ratio (vertical/horizontal) goes to zero \cite{AG,AG2,BL}.
 \medskip

 The unknowns of the problem  are ${\bf U}=({\bf u},u_3):\Omega\times
 (0,T) \rightarrow \R ^3$ the  3D velocity field
 (with   ${\bf u}=(u_1 , u_2 )$ the  horizontal velocity
 and $u_3$ the vertical one) and $p:\Omega\times (0,T)
\rightarrow \R $ the pressure.

\medskip

Also, ${\bf b}({\bf u})= f {\bf u}^\perp$  represents the effect
of the Coriolis Forces, with ${\bf u}^\perp=(-u_2, u_1)^t$ and
$f=2|w|\sin\theta$, where $w$ is the angular velocity of the Earth
and $\theta=\theta(y)$ is the latitude, $\rho\in\R_+$ is the water
density
 (that it is assumed a positive constant), $g\in\R_+$ is the gravity acceleration
 (another positive constant),
  ${\bf f} : \Omega \times (0,T) \rightarrow \R ^2$ is a field of external
horizontal forces (depending for instance on the salinity and
temperature) and ${\bf g}_s:\Gamma_s \times (0,T) \rightarrow \R ^2$
represents the stress of the wind on the surface.

\medskip

Finally,  $\nabla=(\gradH, \partial_z)^t $ stands for  the
three-dimensional gradient operator  (with $\gradH =(\partial_x ,
\partial_y)^t$ its  horizontal
 component) y  $\laplac$ stands for  the  three-dimensional
 Laplacian operator.
 
 The problem (P) have been vertically scaled (see \cite{AG2}) such that the horizontal and vertical dimensions in $\Omega$ are of the same order.

%

\begin{remark}  When variations in the surface are important in the
  problem, it is usual to consider the general  Navier Stokes equations
  with  hydrostatic pressure,  introducing the  free surface as a new unknown. In this case, one has to
 change the boundary condition of ``rigid lid''  ($u_3 = 0$ on
 $\Gamma_s$) for the equation of the free surface, arriving at the
so-called  three-dimensional Shallow Water model.  Some numerical approximations of this model can be seen in  \cite{cas1, cas2, cas4}.  
\end{remark}

We will give two reformulations of  problem $(P)$ leading to different spatial approximations.


The vertical dependence of the pressure can be decomposed by using the hydostatic pressure $\rho g z$,  defining 
 $p_{s}(t;{\bf x})= p(t;{\bf x},z)- \rho g z,$ where 
$p_s: S\times (0,T) \rightarrow \R $ is a new unknown (defined
only on the surface $S$),
  that it will be called   \emph{surface pressure}.

   Notice that  incompressibility equation
$\nabla\cdot{\bf U}=0$ in $\Omega\times (0,T)$ and  boundary
condition $u_3=0$ on $\Gamma_s\times (0,T)$
  are equivalent to the following integral formula for the vertical velocity:
  \begin{equation}\label{vert-veloc}
u_3 (t;{\bf x},z)  = \int_z^0 \gradH \cdot  {\bf u}(t;{\bf x},s)\,ds.
\end{equation}
    Moreover, the following equality holds:
\begin{equation}\label{equality-equiv}
\int_{-D({\bf x})}^0 \nabla \cdot {\bf U} \, dz = \gradH \cdot
\langle {\bf u}\rangle - ({\bf u},u_3)({\bf x},-D({\bf
x}))\cdot(\nabla_{\bf x} D({\bf x}),1)=0\quad \hbox{in $S\times
(0,T)$},
\end{equation}
where $\langle {\bf u}\rangle$ denotes the total vertical flux of
the horizontal velocity:
$$
\langle {\bf u} \rangle (t;{\bf x})=\int_{-D({\bf x})}^{0} {\bf u}
(t;{\bf x},z)\, dz. $$ 
Therefore, since $(\nabla_{\bf x} D({\bf
x}),1)$ is parallel to the normal vector $({\bf n}_{\bf x}, n_3)$ on
$\Gamma_b$, assuming $\nabla\cdot {\bf U}=0$ in $\Omega\times
(0,T)$,
 the so-called slip condition ${\bf u}\cdot
{\bf n}_{\bf x}+ u_3n_3=0$
 on $\Gamma_b\times (0,T)$
 is equivalent to the constraint $\gradH\cdot \langle {\bf
u}\rangle=0$ in $S\times (0,T)$ (\cite{l,ltw,ltw2}).

 Then,  problem (P) can be reformulated
as the following \emph{integro-differential problem}:
$$
\left \{ \begin{array}{rl}
 \partial_t {\bf u} \,+\, ( {\bf U}\cdot \gradV ) {\bf u}\,
- \nu  \Delta {\bf u} + {\bf b}({\bf u})\,+\,
               \gradH\, p_{s} \,
             =
   \, {\bf f} & \hbox{ in $ \Omega \times (0,T)$,} \\
 \noalign{\smallskip}
 \gradH  \cdot \langle {\bf u}\rangle \,  =  \, 0  &\hbox{ in $ S \times
(0,T)$,}  \\
\noalign{\smallskip}
 \nu \partial_z {\bf u}  = {\bf g}_s  &  \hbox{on $ \Gamma _s
   \times (0,T)$,}
   \\
\noalign{\smallskip}
{\bf u}  = 0  & \hbox{ on $ (\Gamma_b\cup\Gamma_l) \times
(0,T)$,}
  \\
\noalign{\smallskip}
 {\bf u}_{\vert t=0}= {\bf u}_0 &  \mbox{ in } \Omega ,
       \end{array} \right.
\leqno{(Q)}
 $$
where ${\bf U}=({\bf u},u_3)$ with  $u_3$
depending on the $\gradH \cdot  {\bf u}$ by the integral formula
(\ref{vert-veloc}).
\medskip

Instead of problem $(Q)$, other reformulation  of  problem $(P)$ can be done, 
based on the following equivalence: assuming
the slip-condition  $ {\bf u} \cdot {\bf n}_{\bf x} + u_3\, n_3 =0 $ on $\Gamma_b$, one has
$$
\left.
\begin{array}{c}
  \partial_z (\nabla \cdot {\bf U})=0 \q \hbox{in $\Omega$}\\
 \gradH \cdot \langle {\bf u} \rangle =0\q \hbox{in $S$}\\
\end{array}
\right\} \Longleftrightarrow \hbox{$\nabla \cdot {\bf U} =0$ in
$\Omega$.}
$$
Indeed, from  the equation   $\partial_z (\nabla  \cdot {\bf U})=
0$, one has  $\nabla \cdot {\bf U}=a({\bf x})$. By integrating in
vertical this equality and using (\ref{equality-equiv}) (taking
into account  $\gradH \cdot  \langle {\bf u} \rangle =0$ and
the slip-condition  on $\Gamma_b$), one has
 $$0= \int_{-D({\bf x})}^0 g({\bf x}) \, dz =D({\bf x})\, a({\bf x}) \quad\hbox{in $S$.}$$
  Then $a\equiv 0$ in $S$ and $ \nabla \cdot {\bf U}=0$.
   Conversely, since $ \nabla \cdot {\bf U}=0$ then
 $\partial_z (\nabla \cdot {\bf U})=0$. Moreover, $\gradH \cdot \langle {\bf
   u} \rangle =0$ is deduced again from (\ref{equality-equiv}) integrating in vertical $ \nabla \cdot
  {\bf U}=0$ and taking into account the slip-condition on $\Gamma_b$.
  
Therefore, the second  reformulation  of  problem $(P)$ is:
$$
\left \{ \begin{array}{rl}
 \partial_t {\bf u} \,+\, ( {\bf U}\cdot \gradV ) {\bf u}\,
- \nu \Delta {\bf u} + {\bf b}({\bf u})\,+\,
               \gradH\, p_{s} \, =
   \, {\bf f} & \hbox{ in $ \Omega \times (0,T)$,} \\
 \noalign{\smallskip}
\partial_z ^2 u_3  + \partial_z \, \gradH \cdot {\bf u} =  0    &
\hbox{ in $ \Omega \times (0,T)$,} \\
\noalign{\smallskip} \gradH \cdot \langle {\bf u} \rangle = 0 &
\hbox{ in $ S \times (0,T)$,}
\\ \noalign{\smallskip}
 \nu \partial_z {\bf u}  = {\bf g}_s  &  \hbox{on $ \Gamma _s
   \times (0,T)$,}
   \\ \noalign{\smallskip}
{\bf u}  = 0  & \hbox{ on $ (\Gamma_b\cup\Gamma_l) \times
(0,T)$,}
\\ \noalign{\smallskip}
  u_3  = 0 & \hbox{ on  $(\Gamma_s\cup\Gamma_b)\times (0,T)$,}
 \\
\noalign{\smallskip}
 {\bf u}_{\vert t=0} = {\bf u}_0 &  \mbox{ in } \Omega .
       \end{array} \right.
\leqno{(R)}
 $$
  Notice that in $(R)$ the vertical velocity $u_3$ is uniquely defined by
  the $z$-elliptic problem 
   \begin{equation}\label{u3-difer-eq}
\partial^2_z u_3=-\partial_z\gradH \cdot{\bf u}\q \hbox{in $ \Omega \times 
(0,T)$,}\q u_3|_{\Gamma_s\cup\Gamma_b}=0.
\end{equation}

From the numerical analysis point of view, the convergence of some
Finite Element (FE) schemes for the stationary problem related to
$(Q)$, has been proved in \cite{cg}, where  the so-called
\emph{hydrostatic Inf-Sup} stability condition appears. To
approximate the time-dependent problem, a stabilized FE scheme was used by Chac\'on-Rodr\'{\i}guez in
\cite{c-rg-04,c-rg-05}, and Bermejo in \cite{bermejo} and
Bermejo-Galán in \cite{bermejo2}  used a
semi-lagrangian projection time-scheme together with finite
elements in space. On the other hand, Chacón-Gómez-Sánchez in \cite{cgs} have derived the numerical approximation of this model by the Orthogonal Sub-Scales FE method, obtaining  stability, convergence and error estimates (optimal for 2D flows) for a steady linearized model and they have performed some numerical tests for the non-linear case.
\medskip

The goal of this paper is to design   numerical schemes associated to both  formulations $(R)$ and $(Q)$, 
based on a fractional-step  time-scheme  and FE in space, which satisfies analytical results into two directions: on one hand, 
unconditional stability and convergence towards weak solutions of
$(R)$ and, on the other hand,  error estimates with respect to a
sufficiently regular solution, under the constraint $$k\le h^2 .
\leqno{\bf (H)}$$ 

These results could be seen as an extension of the numerical  analysis done for a viscosity-splitting scheme applied to the time-dependent Navier-Stokes
Equations in   \cite{bch} and
\cite{G-Re-cras,g-re-NM,g-re-space}. Nevertheless, error estimates for the  Navier-Stokes case  have been deduced  in \cite{G-Re-cras,g-re-space}, using the corresponding time-discrete scheme as intermediate problem to 
obtain error  estimates for  the fully discrete scheme, under the constraint 
$$
h\le C\, k.
$$
Since this constraint has contrary sense that {\bf (H)}, it is not clear how the argument done in \cite{G-Re-cras,g-re-space} for Navier-Stokes  could be extended for the Primitive Equations.
\medskip

 In the scheme studied in this paper,   three subproblems
must be solved at every time step $m$. Indeed, given $({\bf u}_h^{m},p_{s,h}^{m})$,
  firstly  the vertical velocity $u_{3,h}^{m}$ is computed
 in function of  $\gradH \cdot  {\bf u}_h^{m }$, afterwards
   an intermediate horizontal velocity ${\bf  u}_h^{m+1/2}$
   and finally a pair  $({\bf u}_h^{m+1}, p_{s,h}^{m+1})$ is computed solving an  Hydrostatic Stokes  problem.




\medskip

The rest of this paper is organized as follows. After giving  some
preliminaries in Section 1,  the fully discrete
scheme related to $(R)$ is described in Section 2, obtaining in Section 3 some  stability a priori estimates and convergence  (by subsequences) as
$(k,h)\to 0$ towards weak solutions of  $(R)$.

 In Section 4, 
   some error estimates for velocity and  pressure are deduced.
  Firstly, we obtain $O(\sqrt{k}+h^l)$ error estimates
   for both  velocities ${\bf u}_h^{m+1/2}$ and ${\bf u}_h^{m+1}$, improving to optimal accuracy $O(k+h^l)$ for the ``end of step'' velocity ${\bf u}_h^{m+1}$,  for $l=1,2$   the order of the FE approximation. Afterwards,   $O(\sqrt{k}+h^l)$  for the discrete time derivative of ${\bf u}_h^{m+1}$ and  for the pressure will be deduced. 
   
     On the other hand, only when $l=2$, we obtain optimal $O(k+h^2)$ error estimates   for the discrete time derivative of ${\bf u}_h^{m+1}$ and for the pressure.

In order to also deduce optimal accuracy for the pressure  when $l=1$,  in Section 5 we consider a modified   scheme associated to  $(Q)$, where the vertical velocity is approximated via an integral computation like 
(\ref{vert-veloc}). For this scheme, optimal  
 accuracy in the $L^2(\Omega)$-norm; $O(k+h^{l+1})$   for ${\bf u}_h^{m+1}$ and $O(k+h^l)$   for the pressure are obtained. Finally, some  
comments about the treatment of the Coriolis term are given in
Section 6, and some conclusions in Section 7.

\section{Preliminaries}

\subsection{The discrete Gronwall  Lemma}
In this paper, the following discrete Gronwall lemma will be
frequently used (see~\cite[p.~369]{RANNACHER}):

\begin{lemma} \label{GronwallD}
Let $k$, $B$ and $a_m\,$, $b_m\,$, $c_m\,$,
$\gamma_m$ be nonnegative numbers.
\begin{description}
\item[{\bf a)  (Discrete Gronwall inequality)}] We assume
$$
   a_{r+1}+k\sum_{m=0}^{r}b_{m}\leq k\sum_{m=0}^{r}
   \gamma_{m}a_{m}+k\sum_{m=0}^{r}c_{m}+B\qquad\forall r\geq0.
$$
 Then,  one has
$$
   a_{r+1}+k\sum_{m=0}^{r}b_{m}\leq\exp\left(k\sum_{m=0}^{r}
   \gamma_{m}\right)
   \left\{k\sum_{m=0}^{r}c_{m}+B\right\}\qquad\forall r\geq 0.
$$
\item[{\bf b)  (Generalised discrete Gronwall inequality)}] We assume
$$
   a_{r}+k\sum_{m =0}^{r}b_{m }\leq k\sum_{m=0}^{r}
   \gamma_{m}a_{m}+k\sum_{m=0}^{r}c_{m}+B\qquad\forall r\geq0
$$
such that $k\gamma_{m}<1$ for all $m$. Then, setting
$\sigma_{m}\equiv(1-k\gamma_{m})^{-1}$, one has
$$
   a_{r}+k\sum_{m=0}^{r}b_{m}\leq\exp\left(k\sum_{m=0}^{r}\sigma_{m}
   \gamma_{m}\right)
   \left\{k\sum_{m=0}^{r}c_{m}+B\right\}\qquad\forall r\geq 0.
$$
\end{description}
\end{lemma}

\subsection{Space of functions and weak solutions}
To define  the notion  of  weak solution of problem $(R)$, we
introduce  the following Hilbert spaces:
\begin{eqnarray*}
      H_{b,l}^1(\Omega ) &=& \{ {\bf v} \in  H^1(\Omega )
  \ / \ {\bf v}|_{\Gamma_b\cup\Gamma_l}=0 \},
  \\
   {\bf H}&=&\{  {\bf v} \in L^2 (\Omega )^2 \ / \
  \gradH \cdot \langle {\bf v} \rangle =0 \mbox{ in } S,
  \  \langle {\bf v} \rangle \cdot {\bf n}_{\partial S}=0 \},
 \\
   {\bf V}&=& \{ {\bf v} \in H_{b,l}^1(\Omega )^2 \ / \
  \gradH \cdot \langle{\bf v}\rangle =0 \mbox{ in } S \},
\end{eqnarray*}
being ${\bf n}_{\partial S}$ the normal outward unitary vector of
$\partial S$. Observe that spaces $ {\bf H}$ and $ {\bf V}$ are
the ``hydrostatic version'' of the classical spaces used  for the
Navier-Stokes equations.

We denote ${\bf H}_{b,l}^1(\Omega )= H_{b,l}^1(\Omega )^2$, etc. The
norm and scalar product in $L^2 (\Omega) $ will be denoted by $|
\cdot | $ and $ \Big(\cdot, \cdot \Big) $, whereas in
  ${\bf H}_{b,l}^1 (\Omega) $   we denote by  $\| \cdot \| $ the norm of the  gradient in $L^2(\Omega)$,
  that is $\|u\|=|\nabla u|$.
On the other hand, we denote by  ${\bf H}_{b,l}^{-1}(\Omega)$ and
$H^{-1/2}(\Gamma_s)$ the dual spaces of
  ${\bf H}_{b,l}^{1}(\Omega)$ and  $H^{1/2}(\Gamma _s)$
respectively, with  duality products $\langle
\cdot,\cdot\rangle_\Omega$ and
$\langle\cdot,\cdot\rangle_{\Gamma_s}$. Spaces defined in $\Omega $ will be frequently abbreviate as $L^2$ instead of $L^2(\Omega) $ etc.

\medskip

The space related to the surface pressure will be:
$$
L^2_0(S)=\left\{q\in L^2(S)\ /\ \int_S q=0 \right\}.
$$

\medskip

The  vertical velocity $u_{3}$ can be  obtained from $\gradH
\cdot {\bf u}$ by means of either the integral formulation
(\ref{vert-veloc}) or the differential formulation
(\ref{u3-difer-eq}). In this process, the $L^2(\Omega)$ regularity
for the horizontal derivatives of $u_3$ is not obtained, hence the following 
``anisotropic'' Hilbert spaces should be defined:
$$
 H(\partial _z )=\{ v \in L^2 (\Omega ) \ / \
  \partial_z v \in L^2 (\Omega ) \}
  \quad
  (\hbox{resp. } H^k(\partial _z )=\{ v \in H^k (\Omega ) \ / \
  \partial_z v \in H^k (\Omega ) \}),
$$
and $H_0(\partial_z)=\{ v\in H(\partial_z) \ / \ v=0 \mbox{ on
}\Gamma_s\cup\Gamma_b\}$. The inner products are defined by $(v,w)_{H(\partial _z
)}=\Big(v,w\Big)+\Big(\partial_zv,\partial_z w\Big)$ in $H(\partial _z )$ and by $(v,w)_{H_0(\partial _z
)}=\Big(\partial_zv,\partial_z w\Big)$ in
$H_0(\partial _z )$, owing to a vertical
Poincar\'e inequality (see (\ref{dpoincare}) below).

Notice that, given ${\bf u}\in {\bf H}_{b,l}^1(\Omega)$, the
weak solution $u_3$ of problem (\ref{u3-difer-eq}) can be  defined by:
$$u_3\in H_0(\partial _z ) \q \hbox{such that}\q
(u_3,w)_{H_0(\partial _z )}= -\Big(\gradH\cdot{\bf u},\partial_z
w\Big)\q \forall\, w\in H_0(\partial _z ). $$


Due to the loss of regularity of $u_3 $ ($u_3 \in L^2(\Omega)$ but $u_3
\not\in H^1(\Omega)$), the vertical convection term $u_3\partial_z{\bf u}$
does not belong to ${\bf H}_{b,l}^{-1}(\Omega )$, hence 
 more regular test functions must be introduced in the variational
 formulation of $(R)$. For instance, it  suffices to consider
 ${\bf v}\in {\bf H}_{b,l}^1(\Omega )$ such that
 $\partial_z {\bf v} \in {\bf L}^3(\Omega)$,
because in this case one has (see \cite{cg}):
$$
\Big|\Big\langle ({\bf U} \cdot \nabla){\bf u},{\bf v} \Big\rangle_\Omega \Big|
=
 \Big| \int_\Omega({\bf U} \cdot \nabla){\bf v}\cdot{\bf u} \Big|  < +\infty.
$$
Another possibility is to assume ${\bf v}\in {\bf
H}_{b,l}^1(\Omega )\cap {\bf L}^\infty(\Omega)$, and then
$\dis\int_\Omega({\bf U} \cdot \nabla){\bf u}\cdot{\bf v} <
+\infty$.

\medskip

 For fully discrete schemes, it is usual to use the following
skew-symmetric of the
convective terms: for each ${\bf U} \in {\bf H}^1_{b,l} \times
H_0(\partial_z ),\, {\bf v}\in {\bf H}^1$, ${\bf w}\in {\bf H}^1 $
with  either ${\bf w}\in{\bf L}^{\infty}$  or $\partial_z {\bf w}
\in {\bf L}^3 $,
\begin{eqnarray*}
c\Big({\bf U},{\bf v} ,{\bf w}\Big) &=&  \int_\Omega \Big \{ ({\bf
U} \cdot \nabla ) {\bf
  v} \cdot {\bf w} + \frac{1}{2}  ( \nabla \cdot {\bf U} ) {\bf
  v}\cdot {\bf w} \Big \} \q\q
\hbox{if ${\bf w} \in {\bf L}^\infty$}  \\
 &=& - \int_\Omega \Big \{ ({\bf U} \cdot \nabla) {\bf
  w} \cdot {\bf v} + \frac{1}{2}  ( \nabla \cdot {\bf U} ) {\bf
  v}\cdot {\bf w} \Big \} \q \q
\hbox{if $ \partial_z{\bf w}\in {\bf L}^3$}.
\end{eqnarray*}
 Obviously, $c({\bf U},{\bf v} ,{\bf w})=\displaystyle  \int_\Omega({\bf U} \cdot
\nabla ) {\bf v} \cdot {\bf w}$ whether $\nabla \cdot {\bf U}=0$.
 By simplicity, we denote the vertical part of these trilinear forms in the same  manner, i.e.
$$c\Big(u_3,{\bf v} ,{\bf w}\Big)= \int_\Omega \Big \{ u_3 \,\partial_z  {\bf
  v} \cdot {\bf w} + \frac{1}{2}  (\partial_z u_3) \, {\bf
  v}\cdot {\bf w} \Big \}
  = - \int_\Omega \Big \{ u_3 \, \partial_z  {\bf
  w} \cdot {\bf v} + \frac{1}{2}   (\partial_z u_3) \, {\bf
  v}\cdot {\bf w} \Big \}$$
Previous equalities hold even for discrete spaces, hence in the sequel, we
will use any of these two possibilities.

Let us to define the weak solutions of $(R)$ in $(0,T)$,  solving the following 
variational formulation  (in a reduced form without pressure):
  Find ${\bf U}=({\bf u},u_3) \in L^2(0,T;{\bf V} \times
H_0(\partial_z ))$ with ${\bf u}\in L^\infty(0,T;{\bf H})$ and ${\bf u}_{\vert t=0} = {\bf u}_0$ in $\Omega$, 
such that a.e.~$t\in (0,T)$:
$$
\left\{   
\begin{array}{l} \displaystyle 
\Big({\bf u}_t,{\bf v}\Big) + c \Big({\bf U},{\bf u}, {\bf v}\Big) + \nu\Big(\nabla
{\bf u}, \nabla {\bf v}\Big) 
= \Big\langle{\bf f}, {\bf v}\Big\rangle_\Omega +\Big\langle{\bf
g}_s, {\bf v}\Big\rangle_{\Gamma_s},\q \forall\, {\bf v} \in {\bf
V}\cap{\bf W}^{1,3}_{b,l}\cap {\bf L}^\infty,
\\ \Big ( \partial_z \, u_3 , \partial_z w \Big ) = -
\Big ( \gradH \cdot {\bf u} ,  \partial_z w \Big ) ,\q \forall \, w
\in H_0(\partial_z).\end{array} \right.
$$
Here  the Coriolis term has not been considered, because it  does not add important difficulties to the arguments of this paper. In Section~\ref{coriolis}, we briefly analyze some possibilities  to approximate  the Coriolis term.

\subsection{Known analytical results}

The existence of weak solutions $({\bf u},p_s)$ of  problem
$(Q)$ is well known, see Lions-Teman-Wang \cite{ltw2} and
Lewandowski \cite{l}, always in domains with side-walls (i.e.~$D
\ge D_{min}>0$ in $\overline{S}$). In these works,  a compactness
method is used to obtain the velocity ${\bf u}$ in a space with
the restriction $\nabla \cdot \langle {\bf u} \rangle =0$ and
afterwards, the surface pressure $p_s$ is obtained by means of a
specific De Rham's lemma on the surface $S$. In domains without
side-walls (i.e.~when the depth function $D$ can degenerate to
zero), the existence of weak solutions $({\bf u}, u_3, p)$ of
$(P)$ is obtained by an asymptotic limit  applied to the
Navier-Stokes equations with anisotropic viscosity when the ratio
depth over horizontal diameter (of the domain) goes to zero; see
Besson-Laydi \cite{BL} for the stationary case and
Azerad-Guill\'en \cite{AG,AG2} for the time-dependent one. The
existence of weak solutions of the stationary problem related to
$(Q)$ in domains without side-walls is proved in
Chac\'on-Guill\'en \cite{cg} by an internal approximation argument; 
a mixed (velocity-pressure) variational formulation  is
approximated by a conformed Finite Element method verifying the
so-called ``hydrostatic Inf-Sup condition".
On the other hand,  Orteg\'on in \cite{ortegon} obtained a generalization
of De Rham's Lemma to  general domains without side-walls.

\medskip

Respect to  regularity results for the Primitive Equations, the
existence of  strong solutions (with $H^2(\Omega)$-regularity for
the horizontal velocity)  is treated by Ziane in \cite{Z} for the
linear stationary problem associated to $(Q)$. This result is
extended in \cite{gr} to the linear evolutive case. For 
the nonlinear problem, the  existence (and uniqueness)  of local in
time strong solutions for $2D$ domains (global in time for small enough data),  is proved in  \cite{gr}. The
extension (and improvement) of this kind of results to $3D$ domains
can be seen in \cite{gmma}. Finally,  assuming flat bottom and
Neumann boundary condition on the bottom, the existence of global in
time regular solutions without constraints is proved in
\cite{cao-titi}. In \cite{ziane}, this result is also obtained with
 Dirichlet boundary conditions on the bottom.
  \medskip

\subsection{Some 3D anisotropic spaces and related estimates}

 Given $p,q \in [1,+\infty]$, it will be said that
a function $u\in L_z^q L_{\bf x}^p(\Omega)$ (or simply $L_z^q L_{\bf x}^p$) if:
$$
  u (\cdot, z ) \in L^p(S_z) \quad \mbox{and} \quad
  \Vert u(\cdot ,z )
  \Vert_{L^p(S_z)} \in L^q(-D_{\rm max},0),
$$
where $S_z= \{{\bf x} \in S\ :\ ({\bf x},z) \in \Omega \}$, and its
norm is given  by  $
 \left\Vert \Vert u(\cdot ,z) \Vert_{L^p(S_z)}
 \right\Vert_{L^q(-D_{\rm max},0)}.$
Some   anisotropic  norms frequently used in this paper will be:
$$
\begin{array}{rcl}
   \Vert u \Vert_{L_z^2L_{{\bf x}}^4(\Omega)} & = &
   \left( \displaystyle\int_{-D_{\rm max}}^0
   \Vert u(\cdot ,z ) \Vert_{L^4(S_z)}^2 dz  \right)^{1/2} \\
\noalign{\vspace{-1ex}}\\
   \Vert u \Vert_{L_z^{\infty}L_{{\bf x}}^2(\Omega)} & = &
   \displaystyle\sup_{z \in (-D_{\rm max},0)}  \Vert
   u(\cdot ,z)  \Vert_{L^2(S_z)},
\end{array}$$


In a similar way, we define the spaces
$$
H_z^1 L_{\bf x} ^2 \equiv H^1 ( -D_{\rm max}, 0; L^2 (S_z)),\q
L_z^2 H_{\bf x} ^1 \equiv L^2 ( -D_{\rm max}, 0; H^1 (S_z)).
$$
Notice that $H_z^1 L_{\bf x} ^2=H(\partial_z)$.

 On the other hand,  we will use frequently the following inequalities (see
\cite{gmma}):
\begin{itemize}

 \item Horizontal  Gagliardo-Nirenberg inequality (related to $2D$ domains):
$$\begin{array}{l}
\| u\| _{L_z^2\, L_{\bf x}^4} \le C\, |u|^{1/2} \, |\gradH u|
^{1/2}  \quad  \forall\, u \in L_z^2 H_{\bf x} ^1 \ \ \mbox{such
that} \
u|_{\Gamma_b \cup \Gamma_l}=0 ,\\ 
\| u\| _{L_z^2\, L_{\bf x}^4} \le C\, |u|^{1/2} \, \| u \|^{1/2}
\quad  \forall\, u \in H^1.
\end{array}
$$

    \item  Vertical Poincar\'e Inequality (related to $1D$ domains):
\begin{equation}\label{dpoincare}
| v |  \le D_{\rm max}^{1/2}\,  |\partial_z \, v |, \quad
\forall\, v \in H_z^1 L_{\bf x} ^2 \ \ \mbox{such that} \
v|_{\Gamma_b}=0 \ \mbox{ or } v|_{\Gamma_s}=0.
\end{equation}
    \item Vertical Gagliardo-Nirenberg inequality (related to $1D$ domains):
\begin{equation}\label{du_3}
\| v\| _{L_z^{\infty} L_{\bf x}^2} \le C\, (|v| + |v|^{1/2} \,
|\partial_z v| ^{1/2} ) \quad \forall\, v \in H_z^1 L_{\bf x} ^2 .
\end{equation}
Moreover, if $v|_{\Gamma_b}=0$ or $v|_{\Gamma_s}=0$, then $ \| v \|
_{L_z^{\infty} L_{\bf x}^2} \le C\, | v |^{1/2} \, |\partial_z v |
^{1/2} $.

In particular, from (\ref{dpoincare}) and (\ref{du_3}), one has
\begin{equation}\label{du_4}
\| v\| _{L_z^{\infty} L_{\bf x}^2} \le C\, |\partial_z v| , \quad
\forall\, v \in H_z^1 L_{\bf x} ^2\ \ \mbox{such that} \
v|_{\Gamma_b}=0 \ \mbox{ or } v|_{\Gamma_s}=0 .
\end{equation}
\end{itemize}

\section{Description of the scheme}

The time interval  $[0,T]$ is divided into  $M$ subintervals of equal length  $k=T/M$, arising the partition of $[0,T]$, $\{ t_m=m \, k\}_{m=0}^{M}$. For simplicity and without loss of generality, we fix the viscosity constant $\nu=1$.

 We consider a time-scheme, where in each time step $m$, given $
({\bf f} ^m ,{\bf g}_s^m) _{m=1}^M $  approximations of data
 $( {\bf f}, {\bf g}_s ) $  at 
$t=t_m$, a sequence $({\bf u}_h^m , u_{3,h}^m , p_{s,h}^m )_m$
will be computed, as approximations to a regular solution $( {\bf u} , u_3 , p_s )$
 of $(R)$ at $t=t_m$.

\medskip

 We are going to consider a fractional-step scheme,
 splitting the three main difficulties of the problem, namely:
\begin{itemize}
\item the computation of the vertical velocity,
\item
the non linear convective terms, $({\bf U}\cdot \nabla){\bf u} $ (in
particular, the vertical convection $u_3\partial_z {\bf u}$ is more singular 
 than in the Navier-Stokes case),
\item
the restriction  $ \gradH  \cdot \langle {\bf u}\rangle \, =  \, 0$
in $S \times (0,T)$ related to the surface pressure $p_s$.
\end{itemize}
 Indeed, given $({\bf u}_h^{m},p_{s,h}^{m})$,
  firstly  the vertical velocity $u_{3,h}^{m}$ is computed
 as function of  $\gradH \cdot  {\bf u}_h^{m }$, afterwards
  an intermediate horizontal velocity ${\bf  u}_h^{m+1/2}$ is obtained by 
  using  convective terms,
  and finally 
 $({\bf u}_h^{m+1} , p_{s,h}^{m+1})$ is computed solving a linear hydrostatic Stokes  problem (associated to the restriction  $\gradH \cdot \langle {\bf
u}_h^{m+1}\rangle=0$). This method can be called a ``viscosity-splitting'' scheme, because the diffusion terms  are considered in the last two steps (see Sub-steps 1 and 2 below).


\medskip

Let ${\bf X}_h$, $ Y_h$ and $ Q_h$ be  three adequate families of FE spaces which  it will be precised below. Then, the fully discrete scheme is defined as follows:

\medskip

\noindent {\bf Initialization: } Let ${\bf u}_h^0 \in {\bf X}_h $ be
an approximation of ${\bf u}_0$.
\medskip

\noindent
{\bf Step of time} $m+1$:

\begin{description}
\item {\bf Sub-step 0: } Given ${\bf u}_h^{m} \in {\bf X}_h $,
 compute  $ u_{3,h}^{m}\in Y_h$  such  that, 
$$\Big ( \partial_z  u_{3,h}^{m} , \partial_z y_h
\Big ) = - \Big ( \gradH \cdot {\bf u}^m_h ,  \partial_z y_h
\Big ) \quad \forall y_h \in  Y_h.
 \leqno{(S_0)_h^{m}}$$

\item {\bf Sub-step 1: } Given ${\bf U}_h^m=({\bf u}_h^m ,
u_{3,h}^m)\in {\bf X}_h\times Y_h$,
  compute  ${\bf u}_h^{m+1/2}\in {\bf X}_h$  such  that, 
$$
\left\{
\begin{array}{l} \displaystyle
\frac{1}{k}\Big({\bf
u}_h^{m+1/2}-{\bf u}_h^{m},  {\bf v}_h\Big) + c\Big( {\bf U}_h^{m},
{\bf u}_h^{m+1/2},  {\bf v}_h\Big)
 +  \Big(\nabla \,  {\bf u}_h^{m+1/2},  \nabla \, {\bf v}_h \Big)
 \\
 = \Big\langle {\bf f}^{m+1},  {\bf v}_h \Big\rangle_\Omega
 +\Big\langle{\bf g}_s^{m+1},
  {\bf v}_h \Big\rangle_{\Gamma_s} 
  \quad \forall {\bf v}_h \in {\bf X}_h.
\end{array}
\right.
\leqno{(S_1)^{m+1}_h}
$$
\item {\bf Sub-step 2: } Given ${\bf u}_h^{m+1/2}\in {\bf X}_h$,  compute
 $({\bf u}_h^{m+1},p_h^{m+1})\in {\bf X}_h  \times Q_h $, such
 that 
 \end{description}
$$
\left \{ \begin{array}{l} \displaystyle\frac{1}{k}\Big({\bf
u}_h^{m+1}-{\bf u}_h^{m+1/2}, {\bf v}_h \Big)
 +  \Big(\nabla  ({\bf u}_h^{m+1}-{\bf u}_h^{m+1/2}), \nabla  {\bf v}_h\Big)
 - \Big( p_h ^{m+1}, \gradH \cdot  \langle{\bf v}_h\rangle\Big)_S =  0 ,\\
\noalign{\smallskip} \Big(\gradH \cdot \langle{\bf
u}_h^{m+1}\rangle, q_h \Big)_S = 0 
\quad \forall ({\bf v}_h, q_h)  \in {\bf X}_h \times Q_h .
\end{array} \right.
\leqno{(S_2)^{m+1}_h}
$$



A linear $z$-elliptic problem must be  
computed in Sub-step~0,  a decoupled linear
convection-diffusion problem in Sub-step~1 and  a (generalized) Hydrostatic Stokes problem in Sub-step~2, which will be well-defined if a particular {\sl Inf-Sup} stability condition holds, see \textbf{(H1)} below.

On the other hand, in order to do the effective computation of the integrals on the $2D$ surface $S$ given in  $(S_2)_h^{m+1}$, $\Big( p_h ^{m+1}, \gradH \cdot  \langle{\bf v}_h\rangle\Big)_S$,   it will be necessary to use vertically structured grids.

\medskip

 With respect to the consistency of this scheme, adding $(S_1)_h^{m+1}$ and $(S_2)_h^{m+1}$ one has:
$$\left\{
   \begin{array}{l}
\displaystyle \frac{1}{k} \Big({\bf u}_h^{m+1}-{\bf u}_h^{m},{\bf
v}_h\Big) +c \Big({\bf U}_h^{m} ,{\bf u}_h ^{m+1/2}, {\bf
v}_h\Big) + \Big(\nabla \,  {\bf u}_h^{m+1},  \nabla \, {\bf
v}_h\Big)\\ -\Big(p_h^{m+1} , \gradH \cdot \langle {\bf v}_h
\rangle \Big)_S = \Big\langle {\bf f}^{m+1},  {\bf v}_h
\Big\rangle_\Omega+\Big\langle{\bf g}_s^m,
  {\bf v}_h \Big\rangle_{\Gamma_s}
\end{array}\right.
 \leqno{(S)_h^{m+1}}
$$
This formulation will be used to prove the convergence of the
scheme.

Although  viscosity-splitting schemes solve a mixed method which request higher computational cost than classical projection (segregated) schemes, they 
 present some  advantages because viscosity-splitting schemes have not numerical boundary layer for the pressure due to  diffusion terms are also  included in the free-divergence projection step (here Sub-step~2),  which let to impose exact boundary conditions for the velocity.

 On the other hand,   viscosity-splitting schemes improve the numerical
treatment  of Euler (or semi-implicit) mixed schemes \cite{teman3}, because  Sub-step 2  is a  symmetric  problem which can be formulated as a minimization problem, and then it can be approximated by using many numerical optimization solvers, as the Uzawa's method, the Augmented Lagrangian method, etc.  (\cite{handbook-Glowinski}).

\subsection{Choice of adequate  Finite Element spaces}

We restrict ourselves to the case where the surface
domain $S\subset \mathbb{R}^2$ has a polygonal boundary and the
bottom function $D$ is globally continuous and locally $P_1$,
hence the $3D$ domain $\Omega$ has  polygonal boundary. Moreover, the
following hypothesis will be imposed about $\Omega$:
\begin{description}
\item[${\bf (H0)}$] {\sl Regularity of the Domain}: Assume  $\Omega \subset \mathbb{R}^3$
 such that the Hydrostatic Stokes Problem
 has ${\bf H}^2 (\Omega) \times H^1 (S)$ regularity for (horizontal) velocity and
 pressure respectively. For this, the following hypothesis of existence of sidewalls should 
  be imposed (see \cite{Z}):
 $$
D \ge D_{\min} >0\quad \hbox{in $S$.}
 $$
\end{description}


To discretize the domain $\Omega$, let ${\cal T}_h(\Omega)$ be a regular
and quasi-uniform triangulation of $\Omega$ (with elements $K\in
{\cal T}_h(\Omega)$) and  ${\cal T}_h(S)$ its associated
triangulation of $S$ with elements $T\in {\cal T}_h(S)$. Assume that ${\cal T}_h(\Omega)$ is a vertically structured mesh, and then  each element  $K\in {\cal T}_h (\Omega)$ is projected onto only one  element
 $T \in {\cal T}_h (S)$. Some references  about how to construct  vertically structured
meshes can be seen in \cite{cg}, \cite{c-rg-04} and \cite{c-rg-05} by
using the so-called {\it iso-$\sigma$ layers} or in \cite{sal} by
using a $\P_0$ approximation on the bottom.

  We consider three families of FE  spaces:
  ${\bf X}_h \subset {\bf H}_{b,l}^1
(\Omega) $ for the horizontal velocity, $ Y_h \subset  H_0
(\partial_z )$ for the vertical velocity and $Q_h \subset L_0^2 (S)$
for the pressure. Functions of ${\bf X}_h$ are globally continuous,
whereas functions in $Y_h $ must be globally continuous only with respect
to vertical direction and $Q_h$ could be furnished by discontinuous
functions.

\vspace{0.5cm}

The following hypotheses are required about  $({\bf X}_h , Y_h , Q_h)  $:
\begin{description}
\item[${\bf (H1)}$]  $({\bf X}_h ,Q_h)$  satisfies the so called ``hydrostatic {\sl Inf-Sup"} 
condition (\cite{cg}): There exists $ \beta >0 $ (independent of $
h$) such that, for all $ h>0, $
$$
\sup_{{\bf v}_h\in {\bf X}_h \setminus\{0\}}
\frac{\Big(q_h, \gradH \cdot \langle{\bf v}_h \rangle  \Big)_S}{|
\gradH  {\bf v}_h |+ \|\partial_z {\bf v}_h \|_{L^3}
 }
\ge \beta \  \|q_h \|_{L^2(S)}, \quad \forall q_h \in Q_h \setminus\{0\}.
$$
    \item[${\bf (H2)}$]   The following inverse inequalities hold:
    for each ${\bf v}_h\in {\bf X}_h$,
$$\|{\bf v}_h \|_{L^2_z L^4_{\bf x}} \le C\, h^{-1/2} |{\bf
v}_h|,\quad \|{\bf v}_h \|_{L^\infty_z L^4_{\bf x}} \le C\,
h^{-1/2} \|{\bf v}_h\|_{L^\infty_z L^2_{\bf x}},\quad \|{\bf
v}_h\|_{L_z^2 L_{\bf x}^{\infty}} \le C\, h^{-1} |{\bf v}_h| ,
$$
$$\|{\bf v}_h\|_{L^3} \le C\, h^{-1/2} |{\bf v}_h|, \quad \|{\bf
v}_h\| \le C\, h^{-1} |{\bf v}_h| , \quad \|{\bf v}_h\|_{W^{1,6}}
\le C\, h^{-1} \|{\bf v}_h \| .$$
 \item[${\bf (H3)}$] The approximation properties of order $O(h^l)$ (for $l=1$ or $2$):
$$
 h^{-1}| {\bf v}- I_h {\bf v} | +  \| {\bf v}- I_h  {\bf v} \|
  \le C\, h^{l} \, \| {\bf v} \|_{{\bf
H}^{l+1}} \qquad \forall\, {\bf v} \in {\bf H}^{l+1} (\Omega)\cap
{\bf V},
$$
$$
   | {\bf v}- I_h  {\bf v} | \le C\, h^{l} \, \| {\bf v}
\|_{{\bf H}^l} \quad \forall\, {\bf v} \in {\bf H}^l \cap {\bf V},
$$
$$
 \|q- J_h q \|_{L^2(S)} \le C\, h^l \, \| q\|_{H^l(S)} \qquad \forall\, q \in H^l
(S) \cap L_0^2 (S),
$$
\begin{equation}\label{estHz}
 \|v_3- K_h v_3 \|_{H(\partial_z)}  \le C\, h^l \, \|  v_3 \|_{H^{l+1}}
 \quad \forall\, v_3 \in
H^{l+1} (\Omega ) \cap H_0 (\partial_z ),
\end{equation}
 where $(I_h , J_h): {\bf V}\times L^2_0(S) \rightarrow {\bf X}_h \times Q_h $ is the hydrostatic Stokes projector defined as:
$$(I_h {\bf v},J_h q)\in {\bf X}_h \times Q_h : \left\{
\begin{array}{l}
 \Big(\nabla (I_h {\bf v}-{\bf v} ), \nabla \, {\bf v}_h \Big) - \Big(J_h q -q , \gradH \cdot
\langle{\bf v}_h  \rangle \Big)_S = 0 \quad \forall\, {\bf v}_h \in {\bf X}_h , \\
  \Big(\gradH \cdot \langle I_h {\bf v}  \rangle , q_h \Big)_S = 0 \quad \forall\, q_h
\in Q_h ,
\end{array}
\right. $$
and  $ K_h:  H_0 (\partial_z)
\rightarrow Y_h $ is  the $H_0 (\partial_z)$-projector onto $Y_h$ defined as:
$$K_h v_3\in Y_h \ :\ \Big ( \partial_z (K_h v_3 -v_3),
\partial_z y_h \Big )= 0 \, \quad \forall\, y_h \in Y_h .
$$
\end{description}

There are some possibilities  to define  $({\bf
X}_h,Y_h,Q_h)$ satisfying {\bf (H1)-(H3)}. For instance, to
approximate the pressure, we can consider
$$
Q_h = \{q_h\in C^0(\overline{S})\mbox{ : } q_h|_T\in
\P_1[x],\forall\, T\in{\cal{T}}_h(S)\}\cap  L^2_0(S).
$$
 To choose $({\bf X}_h,Y_h)$  there are at least two possibilities (\cite{cg}, \cite{gd}):
\begin{enumerate}
\item \textit{Taylor-Hood} ($O(h^2)$ approximation, $l=2$);  locally $ \P _2[x,z]$ by  tetrahedrons and globally continuous FE.

\item \textit{Mini-element} ($O(h)$ approximation,  $l=1$). Let 
${\cal{P}}(K)=\P _1 [x,z] \oplus\alpha_K\lambda_1\lambda_2\lambda_3
\lambda_4$ with $\alpha_K\in\mathbb{R}$ and $\lambda_i\in \P
_1(K)$ such that $\lambda_i(a_j)=\delta_{ij}$, with $a_j$ the
vertices of the tetrahedron  $K$. Then, we consider
\begin{eqnarray*}
  {\bf X}_h &=&  \{ {\bf v}_h\in C^0(\overline{\Omega})\mbox{ : } {\bf
v}_h|_K\in {\cal{P}}(K),\forall K\in{\cal{T}}_h\}^2\cap {\bf
H}^1_{b,l}(\Omega), \\
  Y_h &=& \{y_h\in C^0(\overline{\Omega})\mbox{ : } y_h|_K\in \P
_1[x,z],\forall\, K\in{\cal{T}}_h(\Omega)\} \cap H_0(\partial_z).
\end{eqnarray*}
\end{enumerate}

For vertically structured grids furnished by  prisms, there are other 
possibilities to choose ${\bf X}_h$. For instance, using a bubble by prism or a bubble by each column of
vertical prisms as in \cite{gd}. Notice that these possibilities
are not stables for  the Navier-Stokes problem. 

On the other hand, considering triangulations where each element is 
 a right prism, other possibilities to choose $Y_h$ are (\cite{sal}):
\begin{itemize}
    \item locally $\P_0[{\bf x}]\otimes \P_1[z]$ and globally $z$-continuous (for $l=1$),
    \item locally $\P_1[{\bf x}]\otimes \P_2[z]$ and globally continuous and $z$-$C^1$ (for $l=2$).
\end{itemize}
For these  anisotropic FE approximations of $Y_h$, we have  not seen in the literature any approximation result like hypothesis (\ref{estHz}).  

\section{Stability and convergence towards weak solutions of $(R)$}

In this Section, we are going to study  stability properties of
 scheme $(S_0)_h$-$(S_2)_h$ and  convergence towards  weak solutions of problem $(R)$. For this,   we will obtain some
 a priori (stability) estimates  that let us pass to the
limit (convergence), where  compactness results must be applied to
``control'' the limit in the (nonlinear) convective terms.

Fixed the (uniform) partition of $[0,T]$ of diameter $k=T/M$: $\{
t_m=mk \}_{m=0}^M$, for a given vector $u=(u^m)_{m=0}^M$ with
$u^m\in X$ ($X$ being a Banach space), let us to introduce the
following notation for discrete in time norms:
$$
\| u \|_{l^2(X)}=\left(k \sum_{m=0}^M \| u^m \|_X^2 \right)^{1/2}
\quad \hbox{and} \quad \| u \|_{l^\infty(X)}= \max_{m=0,\dots ,M}
\| u^m \|_X .
$$

 In this section, we    consider the following weak regularity on the data
 $$
{\bf f}\in  L ^2 (0,T;{\bf H}_{b,l}^{-1}(\Omega)), \quad {\bf g}_s
\in L^2 (0,T;{\bf H}^{-1/2}(\Gamma _s))\q \hbox{and} \q {\bf u}_{0}
\in {\bf H}, \leqno{{\bf (WR)}}
 $$ and we choose
$$
{\bf f}^{m+1}=\frac{1}{k}\int_{t_m}^{t_{m+1}} {\bf f}(t)\, dt,\q
{\bf g}_s^{m+1}=\frac{1}{k}\int_{t_m}^{t_{m+1}} {\bf g}_s(t)\, dt.
$$

\begin{lemma}\label{apriori_est} (Stability)
Assume $({\bf WR})$ and ${\bf (H1)}$. If $({\bf u}_h^0)$  is bounded in
${\bf L}^2$, then the following estimates hold:
\begin{equation}\label{reg-debil-esq}
 \| {\bf u}_h^{m+1} \| _{l^{\infty}(L^2)\cap l^2 (H^1)} +
 \| \umedio_h \| _{l^{\infty}(L^2)\cap l^2 (H^1)} \le C,
\end{equation}
\begin{equation}\label{reg-diferencias}
 \| {\bf u}_h^{m+1}- \umedio_h  \| _{l^{2}(L^2)}+
\|  \umedio_h -{\bf u}_h^m \| _{l^{2}(L^2)}\le C\, k^{1/2},
\end{equation}
 \begin{equation}\label{reg-debil-vvert}
\|  u_{3,h}^{m+1} \| _{l^2 (H(\partial_z)\cap L^\infty_zL^2_{\bf
x})}\le C.
\end{equation}
\end{lemma}
\begin{proof}
Estimates (\ref{reg-debil-esq}) and (\ref{reg-diferencias}) can be deduced  making
$$\Big ( (S_1)_h^{m+1} , \umedio_h \Big ) + \Big (
(S_2)_h^{m+1} ,\uu_h \Big )$$
 and
using that
$$c\Big({\bf U}_h^m,\umedio_h ,\umedio_h \Big)=0 \quad\mbox{and}\quad 
\Big( p_h ^{m+1}, \gradH \cdot  \langle{\bf u}_h^{m+1}\rangle\Big)_S=0
.$$
Indeed, one arrives at
$$ \frac{1}{2\,k} \Big ( |{\bf u}_h^{m+1}|^2 - |{\bf u}_h^{m}|^2 + |\umedio_h- {\bf u}_h^{m}|^2 + |{\bf u}_h^{m+1}- \umedio_h|^2  \Big )
$$
$$
+ \frac{1}{2} \Big ( \| \umedio_h \|^2 +  \|{\bf u}_h^{m+1}\|^2 +  \|{\bf u}_h^{m+1}- \umedio_h\|^2  \Big ) $$
$$= \Big \langle {\bf f}^{m+1} ,\umedio_h \Big \rangle + \Big \langle {\bf g}_s^{m+1} ,\umedio_h \Big \rangle_S. 
$$
Then, adding from $m=0$ to $r$ (with any $r<M$), we obtain the desired estimates  (\ref{reg-debil-esq}) and (\ref{reg-diferencias}).

On the other hand, taking $y_h = 
u_{3,h}^{m+1} \in Y_h$ as test function in $(S_0)_h^{m+1}$, one has $| \partial_z  u_{3,h}^{m+1} |
\le | \nabla_{\bf x}  \cdot \uu_h  |$. Therefore,
$(\ref{reg-debil-vvert})$ is a consequence of   estimate
(\ref{reg-debil-esq}) and  inequality (\ref{du_4}
).
\end{proof}

\medskip

 Now, we define the following sequences of functions (defined for all
 $t\in [0,T])$):
\begin{itemize}
\item[$\diamond$]
 ${\bf u}_{k,h}^{(i)} : [0,T]\rightarrow {\bf H}_{b,l}^1 (\Omega)$,
  such that
${\bf u}_{k,h}^{(i)}(t) ={\bf u}_h^{m+i/2}$
  if $t\in(t_m,t_{m+1}]$, $i=0,1,2$.
\item[$\diamond$]
$u_{3,k,h}^{(0)} : [0,T]\rightarrow L^2 (\Omega)$,
  such that
 $ u_{3,k,h}^{(0)}(t) = u_{3,h}^{m}$
  if $t\in(t_m,t_{m+1}]$.
\item[$\diamond$]
${\bf u}_{k,h}:[0,T] \rightarrow {\bf H}_{b,l}^1 (\Omega)$,
 continuous, linear by  subintervals and
${\bf u}_{k,h}(t_m)={\bf u}_h^m$.
\end{itemize}

\begin{theorem}[Convergence] \label{convergencia}
Assume ${\bf (WR)}$ and ${\bf (H0)}$-${\bf (H2)}$, then there exists
a subsequence  $(k',h')$ of $(k,h)$, with $ (k',h') \downarrow 0 $,
and a weak solution ${\bf U}=({\bf u},u_3)$
 of $(R)$ in $(0,T)$, such that:
$({\bf u}_{k',h'}^{(i)})$ (for each $i=0,1,2$) and $({\bf
u}_{k',h'})$ converge to  ${\bf u}$
 strongly  in $L^2 (0,T;{\bf L}^2(\Omega))$, weakly-star in
$L^{\infty} (0,T;{\bf L}^2(\Omega))$ and weakly in $L^2 (0,T;{\bf
H}_{b,l}^1(\Omega))$, whereas   $(u_{3,k',h'}^{(0)})$ converges to
$u_3$ weakly in $L^2(0,T;H_0(\partial_z))$.
\end{theorem}

\begin{proof}
Owing to definition of the functions ${\bf
u}_{k,h}^{(i)},u_{3,k,h}^{(0)}$ and ${\bf u}_{k,h}$, Lemma
\ref{apriori_est} says:
$$({\bf u}_{k,h}^{(i)})_{k,h} \mbox{  and } ({\bf
u}_{k,h})_{k,h} \ \mbox{ are bounded in  } \
 L^{\infty}({\bf L}^2)\cap L^2({\bf H}_{b,l}^1),
\q \forall\, i=0,1,2, $$
\begin{equation}\label{estim-vert-veloc}
\label{1b} (u_{3,k,h}^{(0)})_{k,h}  \
 \mbox{ is bounded in  } \ L^2(H_0(\partial_z)).
\end{equation}

\medskip

On the other hand, from 
(\ref{reg-diferencias}), there exists
$C=C(\nu, {\bf u}_0 , {\bf
  f},{\bf g}_s )>0$
 such that, $ \forall \,i,j=0,1,2, $
\begin{equation}
\label{3}
\|{\bf u}_{k,h}^{(i)}-{\bf u}_{k,h}^{(j)}\|
 _{L^{2}({\bf L}^2 )}^2
\leq  \, C \, k,\quad
\|{\bf u}_{k,h}^{(i)}-{\bf u}_{k,h}\|
 _{L^{2}({\bf L}^2)}^2 \leq  \, C \, k .
\end{equation}

Therefore, there exist  subsequences of
 $({\bf u}_{k,h}^{(i)})_{k,h} \mbox{  and } ({\bf
u}_{k,h})_{k,h}$ (denoted in the same way) and a limit function
${\bf u}$ 
verifying the
following weak convergences as $(h,k)\rightarrow 0$:
$$ \begin{array}{c}
({\bf u}_{k,h}^{(i)})_{k,h} \rightarrow {\bf u}, \quad ({\bf
u}_{k,h})_{k,h} \rightarrow {\bf u} \quad \mbox{ in } \left \{
\begin{array}{l}
L^2 (0, T; {\bf H}_{b,l}^1 (\Omega))-\mbox{weak}\\
L^{\infty} (0, T; {\bf L}^2 (\Omega))-\mbox{weak}*
\end{array}\right.
\end{array}
$$
Note that, thanks to (\ref{3}), the uniqueness of the limits ${\bf
u}_{k,h}^{(i)} = {\bf u}$, for each $i=0,1,2$ hold.

Moreover, from (\ref{estim-vert-veloc})
$$
u_{3,k,h}^{(0)}\to u_3\q \hbox{in $L^2(H_0(\partial_z))$.}
$$

Finally, $(S)_h^{m+1}$ can be rewritten (eliminating the pressure)
as follows:
\begin{equation}
\label{5}
\Big( \partial_t  {\bf u}_{k,h},{\bf v}_h \Big)  +c \Big({\bf
U}_{k,h}^{(0)},{\bf u}_{k,h}^{(1)},{\bf v}_h  \Big) +\Big(\nabla{\bf
u}_{k,h}^{(2)},\nabla {\bf v}_h \Big) = \Big({\bf f}^{m+1} ,{\bf
v}_h \Big)+ \Big({\bf g}_s^{m+1} ,{\bf v}_h \Big)_{\Gamma_s} 
\end{equation}

for each ${\bf v}_h \in {\bf X}_h\cap{\bf V}$.

On the other hand, $(S_0)_h^{m}$ is rewritten as
\begin{equation}\label{55}
\Big(\partial_z u_{3,k,h}^{(0)}, \partial_z w_h \Big)= -\Big(\gradH
\cdot {\bf u}_{k,h}^{(0)} , \partial_z w_h \Big) \quad \forall \,
w_h \in Y_h
\end{equation}

To take  limits in  (\ref{5}), we need for instance  compactness of
 $({\bf u}_{k,h}^{(1)})_{k,h}$ in $L^2({\bf L}^2(\Omega))$.
But,  owing to (\ref{3}), it suffices to obtain compactness of
$({\bf u}_{k,h}^{(2)})_{k,h}$ in $L^2({\bf L}^2(\Omega))$. Assuming
this compactness, taking   (\ref{5})-(\ref{55}) in $(k',h')$, the
pass to the  limit when
  $(k',h') \rightarrow 0$ can be realized by a standard way, concluding that
  (${\bf u},u_3)$ is a weak solution  of the continuous problem $(R)$.

Therefore, it suffices to  get compactness of $({\bf
u}_{k,h}^{(2)})_{k,h}$ in $L^2 ({\bf L}^2 (\Omega))$. For this, let us introduce
$${\bf V}_h =\{ {\bf v}_h \in {\bf X}_h \ /\ \Big(\gradH \cdot
\langle {\bf v}_h \rangle , q_h \Big)_{S} =0,\,  \forall q_h \in Q_h
\}$$
  and  the operator $A_h^{-1}: {\bf V}_h \rightarrow {\bf V}_h $
  defined as the inverse of the discrete ``hydrostatic''
 Stokes problem: given
${\bf u}_h \in {\bf V}_h$,   
\begin{equation}\label{prob-hydrostatic}
 A_h^{-1}{\bf u}_h\in
{\bf V}_h \mbox{ such that } \, \, \Big( \nabla \,  A_h^{-1}{\bf
u}_h , \nabla{\bf v}_h \Big)
   = \Big({\bf u}_h , {\bf v}_h \Big)  \quad \forall\, {\bf v}_h \in
{\bf V}_h
\end{equation}
Taking ${\bf v}_h =
A_h^{-1}{\bf u}_h   \in {\bf V}_h$ in (\ref{prob-hydrostatic}),
\begin{equation}\label{dual-norms}
\quad |\nabla \, A_h^{-1}{\bf u}_h |^2  = \Big({\bf u}_h
,A_h^{-1}{\bf u}_h \Big) \le C \, \| {\bf u}_h \|_{V_h^{'}} \, | \nabla A_h ^{-1} {\bf u}_h
|, \quad \forall\, {\bf v}_h \in {\bf V}_h 
\end{equation}
hence
$$ |\nabla A_h^{-1} {\bf u}_h | \le C\,  \| {\bf u}_h \|_{V_h^{'}}.$$
To obtain the inverse bound, we take any ${\bf v}_h \in {\bf V}_h $ in
(\ref{prob-hydrostatic}), then
$$\Big({\bf u}_h, {\bf v}_h \Big)= \Big(\nabla A_h ^{-1} {\bf u}_h , \nabla {\bf
  v}_h \Big) \le | \nabla A_h ^{-1} {\bf u}_h | \, |\nabla {\bf v}_h |
\qquad \forall {\bf v}_h \in {\bf V}_h, $$ hence
$$ \| {\bf u}_h \| _{V_h^{'}}  \le |\nabla \,  A_h^{-1} \, {\bf
    u}_h |.$$

Now,  to obtain compactness of $ {\bf u}_{k,h}^{(2)}$ we follow an
argument of  \cite{g-j}. First of all, we prove  the following
result.
\begin{lemma}  \label{est-deriv-t}
Under hypothesis of Theorem \ref{convergencia}, one has
\begin{equation}\label{fraccionaria-est}
\int_0^{T-\delta} \| {\bf u}_{k,h}^{(2)} (t+\delta) -{\bf u}_{k,h}^{(2)}
(t) \| _{V_h^{'}}^2 \, dt \le C\, \delta , \q \forall \,\delta :
0<\delta <T,
\end{equation}
where $ C>0$ is independent of $k,h$ and $\delta$.
\end{lemma}
\begin{proof}
Since ${\bf u}_{k,h}^{(2)}$ is a piecewise constant function, it
suffices to suppose that $\delta$ is proportional to time step $k$,
i.e., $\delta=r\, k$ for any $r=1,..., M$. In fact, to obtain
(\ref{fraccionaria-est}),  it suffices to prove
\begin{equation}\label{est-constante}
k \sum_{n=0}^{M-r} \| {\bf u}_h^{n+r}- {\bf u}_h^n \|_{V_h^{'}}^2
\le C (r \, k)  , \q \forall \, r=1,\dots,M.
\end{equation}

Taking  $k{\bf v}_h$, for any
${\bf v}_h\in {\bf V}_h$, as test functions in $(S_1)_h^{m+1}$ and adding from $m=n$ to $n+r-1$, 
\begin{equation}\label{constante2}
\begin{array}{rcl}
\displaystyle\Big({\bf u}_h ^{n+r} - {\bf u}_h ^n , {\bf v}_h \Big)&
=&\displaystyle -k\, \sum_{m=n}^{n+r-1} c\Big({\bf U}_h^m,
\umedio_h, {\bf v}_h \Big) - k \, \sum_{m=n}^{n+r-1} \Big(\nabla
\uu_h, \nabla {\bf v}_h \Big)
\\
& +& k\, \displaystyle\sum_{m=n}^{n+r-1} \left\{ \Big(p_h^{m+1} ,
\gradH \cdot \langle {\bf v}_h \rangle \Big)_S + \Big\langle {\bf
f}^{m+1},  {\bf v}_h \Big\rangle_\Omega+\Big\langle{\bf g}_s^m,
  {\bf v}_h \Big\rangle_{\Gamma_s} \right\}. \end{array}
\end{equation}

Now, taking  ${\bf v}_h =k A_h ^{-1} ( {\bf
u}_h^{n+r}- {\bf u}_h^n)$ as  test function in (\ref{constante2}), using
(\ref{dual-norms}) and
 adding from $n=0$ to $M-r$, 
$$\begin{array}{lcl}
k\, \displaystyle\sum_{n=0}^{M-r} \|{\bf u}_h ^{n+r} - {\bf u}_h ^n \|_{V_h'}^2
&=& -k^2\,\displaystyle \sum_{n=0}^{M-r}\,\displaystyle \sum_{m=n}^{n+r-1} c\Big({\bf
U}_h^m, \umedio_h, A_h^{-1} ({\bf u}_h ^{n+r} - {\bf u}_h ^n )
\Big)
\\
&- & k^2 \, \displaystyle  \sum_{n=0}^{M-r}\, \displaystyle \sum_{m=n}^{n+r-1} \Big(\nabla
\uu_h,
\nabla \, A_h^{-1} ({\bf u}_h ^{n+r} - {\bf u}_h ^n ) \Big) \\
 & +& k^2 \displaystyle\sum_{m=n}^{n+r-1} \left\{ \Big\langle
{\bf f}^{m+1}, A_h^{-1} ({\bf u}_h ^{n+r} - {\bf u}_h ^n)
\Big\rangle_\Omega+\Big\langle{\bf g}_s^{m+1},
  A_h^{-1} ({\bf u}_h ^{n+r} - {\bf u}_h ^n) \Big\rangle_{\Gamma_s} \right\}
  \\ \nonumber
  &:=&J_1 +J_2+J_3 .
\end{array}
$$
Now, we have to bound the $J_i$ terms. The bound of $J_3$ is rather
standard. Since $J_2$ is easier to bound than $J_1$, we only
analyze the  more complicate term of $J_1$, which is the vertical convection:
\begin{eqnarray*}
 c \Big(u_{3,h}^m,
\umedio_h, A_h^{-1} ({\bf u}_h ^{n+r} - {\bf u}_h ^n)\Big) &=&
-\frac{1}{2}\Big(\partial_z  u_{3,h}^m\, \umedio_h, A_h^{-1} ({\bf
u}_h
^{n+r} - {\bf u}_h ^n)\Big) \\
   &-& \Big(u_{3,h}^m \,\umedio_h,\partial_z A_h^{-1} ({\bf u}_h
^{n+r} - {\bf u}_h ^n) \Big) =I_1+I_2
\end{eqnarray*}

Since bound $I_1$ is easier  than $I_2$ (in fact, the $I_1$ term
is the classical isotropic term  appearing in the Navier-Stokes
framework, see \cite{g-re-NM}), we only bound $I_2$ using inequality (\ref{du_4})  as follows:
$$I_2 \le \|u_{3,h}^m\|_{L_z^{\infty}  L_{\bf x}^2}\|\umedio_h\|_{L_z^2 L_{\bf x}^6}
\|\partial_z A_h^{-1} ({\bf u}_h ^{n+r} - {\bf u}_h
^n)\|_{L_z^2L_{\bf x}^3} \le C\,\| {\bf u}_h^m  \| \, \| \umedio_h\|
\, \|  A_h^{-1} ({\bf
  u}_h ^{n+r} - {\bf u}_h ^n )\| _{W^{1,3}}$$
To bound the term $\|  A_h^{-1} ({\bf
  u}_h ^{n+r} - {\bf u}_h ^n )\| _{W^{1,3}}$, we are going to use 
the following Lemma  (see Appendix for a proof):
\begin{lemma}\label{regw16} Assuming {\bf (H0)} and the inverse
 inequality $\|{\bf v}_h\|_{W^{1,6}}\le C\frac{1}{h}\|{\bf v}_h\| $ given in {\bf (H2)}, then
 $$ \| A_h ^{-1} {\bf v}_h \|_{W^{1,6}} \le C\, |{\bf
  v}_h |, \quad \forall\, {\bf v}_h\in {\bf V}_h.$$
 \end{lemma}
In particular,  $\| A^{-1} ({\bf
  u}_h ^{n+r} - {\bf u}_h ^n ) \|_{W^{1,3}} \le C | {\bf
  u}_h ^{n+r} - {\bf u}_h ^n |$ hence the following bound holds:
$$J_1 \le C\,k^2  \sum_{n=0}^{M-r}\, \sum_{m=n}^{n+r-1}\| {\bf u}_h^m  \| \,
 \| \umedio_h\| \, | {\bf u}_h ^{n+r} - {\bf u}_h ^n |
$$
Then, applying Fubini's discrete rule, we obtain
$$ J_1\le C\, k^2 \sum_{m=0}^{M-1} \| {\bf u}_h ^m \| \, \| \uu_h \| \,
\sum_{n=\overline{m-r+1}} ^{\overline{m}}  |  {\bf
  u}_h ^{n+r} - {\bf u}_h ^n |
$$
where
$$\overline{m}=
\left  \{
\begin{array}{lr}
0 & \mbox{ if } m<0, \\
m & \mbox{ if } 0\le m\le M-r,  \\
M-r & \mbox{ if } m>M-r.
\end{array}
\right.
$$
Since $ |\overline{m} -\overline{m-r+1} | \le r$, then
$\sum_{n=\overline{m-r+1}} ^{\overline{m}}  |  {\bf
  u}_h ^{n+r} - {\bf u}_h ^n | \le C\, r$. Finally,
  since $k \sum_{m=0}^{M-1} \| {\bf u}_h ^m \| \, \| \uu_h \|\le
  C$, one arrives at  $J_1
\le C \,(r\, k)$. On the other hand, one also has $J_2+J_3 \le C\,
(r\, k)$ and the proof of Lemma~\ref{est-deriv-t} is finished.
\end{proof}


Note that the bound for the fractional derivative in time (\ref{fraccionaria-est}) 
has been obtained in the norm ${\bf V}'_h$ which moves
with respect to the space parameter $h$. But,  the compactness
results (see for instance J. Simon \cite{simon}) does not work in
these conditions. Then, we will use the already cited argument
of \cite{g-j} in order  to find a fixed norm where the time
fractional derivative can be bounded. For this, we consider the
orthogonal projector onto ${\bf V}$: 
$$R_h :{\bf V}_h \rightarrow {\bf V} \hbox{
defined as } \Big(\nabla (R_h {\bf v}_h -  {\bf v}_h) , \nabla {\bf
w}\Big)=0,\quad \forall\, {\bf w} \in {\bf V},$$
 which 
has the following properties (arguing  as in
\cite{g-j}):
$$\| R_h {\bf u}_h \|_{H^1} \le \| {\bf u}_h\|_{H^1} \q (\mbox{
 $H^1$-stability}),$$
$$ |R_h {\bf u}_h -{\bf u}_h|\le C\, h \| \gradH \cdot \langle{\bf
  u}_h \rangle \|_{L^2 (S)}
\q (\mbox{$L^2$-error estimate}),
$$
 and
$$ \|R_h {\bf u}_h \| _{V'} \le \|{\bf u}_h \|_{V_h'} + C\, h.$$
For the second estimate, the $H^2$ regularity of the hydrostatic Stokes
 problem with second member $R_h {\bf u}_h -{\bf u}_h$ must be 
used, and for  the last estimate, it uses the orthogonal projector onto ${\bf V}_h$
$P_h :{\bf V} \rightarrow {\bf V}_h$ defined as
  $\Big( P_h {\bf v} -{\bf v} , {\bf v}_h\Big)=0,\ \forall \,{\bf v}_h \in {\bf V}_h$
   (see \cite{g-j} for more details).

 From here,
using (\ref{est-constante}), one has
$$k \sum_{n=0}^{M-r} \| R_h ({\bf u}_h ^{n+r}-{\bf u}_h^n )\|_{V'}^2
\le C\,k \,\sum_{n=0}^{M-r} \| {\bf u}_h ^{n+r}-{\bf u}_h^n
\|_{V_h'}^2 + C\, h  \le C ( r\, k +h).$$ The above inequality can
be written as
$$\int_0^{T-\delta} \| R_h \, {\bf u}_{h,k}^{(2)} (t+\delta) -
 R_h \, {\bf u}_{h,k}^{(2)} (t) \|_{V'}^2 \, dt \le C\,( \delta + h).$$
Now, we can apply a compactness (by perturbations) result due to P.~Az\'erad and F.~Guill\'en \cite{AG2}, obtaining
 that $ R_h {\bf u}_{k,h}^{(2)} \rightarrow {\bf u}$ in  $L^2(0,T; {\bf
  L}^2)$-strong. From here, arguing again as in  \cite{g-j},  one can deduce 
 ${\bf u}_{k,h}^{(2)} \rightarrow {\bf u}$ in  $L^2(0,T; {\bf
  L}^2)$-strong, and the proof of Theorem~\ref{convergencia} is finished.
\end{proof}

\section{Error estimates with respect to problem $(R)$}
 In this section, we will obtain  optimal error estimates (for the velocity
 and pressure) with respect to  a sufficiently
 regular  solution  $\{{\bf u}, u_3 ,p_s \}$ of problem  $(R)$.


In order to obtain these error estimates, the following constraint between the time step  $k$ and the mesh size $h$ will be  assumed:  
$$ k \le h^2 .
 \leqno{\bf (H)}$$
\subsection{Regularity hypotheses}

 The  following regularity hypotheses for the solution
 $({\bf U}=({\bf u}, u_3),p_s)$ of $(R)$ will be imposed:
\begin{itemize}
\item To obtain order $O(\sqrt{k}+h^l)$  in $l^{\infty}
({\bf L}^2)\cap l^2 ({\bf H}^1)$  for both velocities:
$$  {\bf U}\in L ^{\infty}( {\bf H}^{l+1}), \quad {\bf u}_t \in   L
^{2}( {\bf H}^{l}),  \quad  {\bf U}_t \in L^2({\bf L}^2) , 
  \quad    p_s \in
L^2 (H^l), \ \sqrt{t} \,
   {\bf u}_{tt} \in L ^2 ({{\bf H}_{b,l}^{-1}}).
   \leqno{\bf (R1)}$$

\item
To obtain order $O(k+h^l)$  in
$l^{\infty} ({\bf L}^2)\cap l^2 ({\bf H}^1)$ for the end-of-step velocity:  
$$
{\bf u}_{tt} \in L^{2} ({\bf V'}) . \leqno{\bf (R2)} $$
   \item To get order $O(\sqrt{k}+h^l)$  in $l^{2} ({\bf L}^2)$ for the  time discrete
derivative of end-of-step velocity, in $l^{\infty} ({\bf H}^1)$ for
end-of-step velocity  and  in $ l^2(L^2(S))$ for
pressure, 
 $$
 {\bf U}_{t} \in L^{2} ({\bf H}^{1}),
\q
{\bf u}_{tt} \in L^{2} ({\bf L}^{2}). \leqno{\bf (R3)}
$$
\item To get order $O(\sqrt{k}+h^2)$ ($l=2$) in $l^{\infty} ({\bf L}^2)
 \cap l^2 ({\bf H}^1)$ for the time discrete
derivative of velocities,
  $$\begin{array}{c}     \partial_t p_s\in 
 L^2(H^l),\,
 {\bf U}_{t} \in L^{\infty} ({\bf L}^{3})
\cap L^2 ({\bf H}^{l+1}),  \q {\bf u}_t \in L^{\infty}({\bf H}^2 ),\,
  {\bf u}_{tt}\in L^2({\bf H}^l), \\
\,
{\bf U}_{tt} \in L^{2} ({\bf L}^{2}),\,   \sqrt{t} {\bf u}_{ttt} \in
L^{2} ({\bf H}_{b,l}^{-1}). \end{array}
\leqno{\bf (R4)}$$
  \item To obtain order $O(k+h^2)$ ($l=2$)  in  $l^{\infty} ({\bf L}^2)
 \cap l^2 ({\bf H}^1)$ for the time discrete derivative of end-of-step velocity and in $ l^2(L^2(S))$ for the pressure,
$$ {\bf U}_t \in L^{\infty} ({\bf H}^1), \quad  {\bf u}_{ttt} \in L^{2} ({\bf V}').
 \leqno{\bf (R5)} $$

\end{itemize}

\subsection{Problems related to the spatial errors}

We will present  an error analysis for the fully discrete scheme
$(\umedio_h , \uu _h , p_h^{m+1})$ as an approximation of $( {\bf u}
(t_{m+1}), {\bf u} (t_{m+1}), p(t_{m+1})$. Consequently, we consider
the following errors:
$$\emedio  ={\bf u}(t_{m+1}) -\umedio_h, \qquad \ee ={\bf u}(t_{m+1})  -\uu _h ,
\qquad e_{p}^{m+1} = p_s(t_{m+1}) - p_h^{m+1} ,$$
$$
e_3^{m+1} = u_3 (t_{m+1}) - u_{3,h}^{m+1}, \qquad {\bf E}^{m+1} =
(\ee , e_3^{m+1}).
$$

These errors can be decomposed   as follows (splitting
  interpolation and discrete parts):
$$\emedio  = \ee_i   + \emedio_h  , \qquad \ee =
\ee_i +\ee_h , \qquad e_{p}^{m+1} = e_{p,i}^{m+1} +
e_{p,h}^{m+1} ,$$
$$ e_3^{m+1} =e_{3,i}^{m+1} +e_{3,h}^{m+1} , \qquad
{\bf E}^{m+1}=  {\bf E}^{m+1}_h +{\bf E}^{m+1}_i .$$
  Concretely
$$  \ee_i = {\bf u} (t_{m+1}) -I_h {\bf u} (t_{m+1})
 \q\mbox{and} \q
 \ee_h = I_h {\bf u } (t_{m+1})  -\uu_h ,
$$
$$
\emedio_h = I_h {\bf u} (t_{m+1}) -\umedio_h ,
$$
$$  e_{p,i}^{m+1} = p(t_{m+1}) -  J_h p (t_{m+1})
\q\mbox{and} \q
  e_{p,h}^{m+1} = J_h p(t_{m+1}) - p_h^{m+1} ,
$$
$$ u_{3,i}^{m+1} = u_3(t_{m+1}) -  K_h u_3 (t_{m+1})
 \q\mbox{and} \q
 e_{3,h}^{m+1} = K_h u_3(t_{m+1}) - u_{3,h}^{m+1} ,
$$
 $${\bf E}^{m+1}_i = ( \ee_i , e_{3,i}^{m+1})
  \q\mbox{and} \q
{\bf E}^{m+1}_h = (\ee_h , e_{3,h}^{m+1}) .$$

 Using  the following Taylor expansion with integral rest of a function $\phi =
\phi (t)$:
$$
\phi(t+k)-\phi (t)=\phi'(t+k) \, k - \int_t ^{t+ k} (s-t)
\phi''(s)\, ds,
$$
 and the variational problem  verified for an exact solution $({\bf
U},p_s)$ at $t=t_{m+1}$ of $(R)$, one has:
$$
\left\{   \begin{array}{l} \displaystyle
\Big(\displaystyle\frac{1}{k}({\bf u}(t_{m+1})-{\bf u}(t_m)),{\bf
v}\Big) + c \Big({\bf U}(t_{m}),{\bf u}(t_{m+1}), {\bf v}\Big) +
\Big(\nabla {\bf u}(t_{m+1}), \nabla {\bf v}\Big) 
\\
-\Big(p_s(t_{m+1})
,\gradH \cdot \langle {\bf v} \rangle \Big)_S
= \Big\langle{\bf f}(t_{m+1}), {\bf v}\Big\rangle_\Omega
+\Big\langle{\bf g}_s(t_{m+1}), {\bf v}\Big\rangle_{\Gamma_s}
+\Big(\mathcal{E}^{m+1},{\bf v}\Big) ,\q \forall\, {\bf v}\in{\bf
W}_{b,l}^{1,3}\cap{\bf L}^\infty,
\\
\Big(\gradH \cdot \langle {\bf u}(t_{m+1}) \rangle,q \Big)_S=0 ,\q
\forall\, q\in L^2_0(S),
\\ \Big ( \partial_z \, u_3(t_{m+1}) , \partial_z v_3 \Big ) = -
\Big ( \gradH \cdot {\bf u}(t_{m+1}) ,  \partial_z v_3 \Big ) ,\q
\forall \, v_3\in H_0(\partial_z) ,\end{array}
\right.\leqno{(R)^{m+1}}
$$
where $\mathcal{E}^{m+1}:=-\displaystyle\frac{1}{k}
\displaystyle\int_{t_m}^{t_{m+1}} (t-t_m) \,   {\bf u}_{tt} (t)\, dt
  -\dis\left( \int_{t_m}^{t_{m+1}}  {\bf U}_t\cdot
    \gradV\right) {\bf u} (t_{m+1})$ is the consistency error.

For simplicity, we assume
${\bf f}\in C([0,T];{\bf H}_{b,l}^{-1})$ and
 ${\bf g}_s\in C([0,T];{\bf H}^{-1/2}(\Gamma_s))$ and we choose
$$
 {\bf f}^{m+1}={\bf f} (t_{m+1}) \quad \mbox{ and } \quad
 {\bf g}_s^{m+1} ={\bf g}_s (t_{m+1}).
$$
Then,  the data errors $ {\bf f} (t_{m+1})-{\bf f}^{m+1}$ and
${\bf g}_s (t_{m+1})-{\bf g}_s^{m+1}$  vanish.

Comparing $(R)^{m+1}$ with $(S_0)^{m}_h$ and $(S_1)^{m+1}_h$,
 the following variational problems
for the spatial errors $e_{3,h}^m$ and $\emedio_h$ hold:
$$\Big(\partial_z  e_{3,h} ^{m} , \partial_z y_h
\Big)  =
-\Big( \gradH \cdot ({\bf e}_h^m + {\bf e}_i^m), \partial_z
y_h
\Big)\quad \forall y_h
\in Y_h , \leqno{(E_0)_h ^{m}}
$$
$$
 \left \{ \begin{array}{l}
\dis\frac{1}{k}\Big(\emedio_h - {\bf e}_h^{m} , {\bf v}_h \Big)
 + \Big( \nabla\,   \emedio_h , \nabla\,  {\bf v}_h \Big)-
 \Big( p_s(t_{m+1}), \gradH\cdot\langle{\bf v}_h\rangle \Big)_S\\
= \dis{\bf NL}^{m+1}({\bf v}_h)  + \Big(\mathcal{E}^{m+1} , {\bf
v}_h \Big)
-\Big(\delta_t\ee_i, {\bf v}_h
\Big)-\Big(\nabla \ee_i, \nabla {\bf v}_h \Big) , \quad \forall\,
{\bf v}_h \in {\bf X}_h ,
\end{array} \right.
\leqno{(E_1)^{m+1}_h}
$$
 where $\delta_t\ee_i=\displaystyle\frac{1}{k}\Big(\ee_i - {\bf e}_i^{m}\Big)$
 and
$${\bf NL}^{m+1}({\bf v}_h)= -c\Big({\bf E}_h^m+{\bf E}_i^m  , {\bf u}(t_{m+1}) ,
{\bf v}_h \Big)- c\Big( {\bf U}_h^m , \emedio_h + \ee_i, {\bf v}_h
\Big).
$$
On the other hand, adding and substrating $I_h{\bf u}(t_{m+1})$ to
$(S_2)_h^{m+1}$ one has 
$$
\left \{ \begin{array}{l} \displaystyle\frac{1}{k}\Big(  {\bf
e}_h^{m+1}- \emedio_h , {\bf v}_h \Big)
  +  \Big( \nabla  ( \ee_h -\emedio_h), \nabla {\bf v}_h \Big)
= - \Big( p_h^{m+1} ,\gradH \cdot  \langle{\bf v}_h \rangle \Big)_S
     \\
\noalign{\smallskip} \Big(\gradH \cdot \langle \ee_h \rangle, q_h
\Big)_S=0 ,
\end{array} \right.
\leqno{(E_2)^{m+1}_h}
$$
for each $({\bf v}_h,q_h)  \in {\bf X}_h\times Q_h$. 

Due to the choice of the projector $K_h$, $\Big(\partial_z  e_{3,i}
^{m+1} , \partial_z  y_h \Big)=0 $, hence it is not appear in $(E_0)_h^{m+1}$. Since the same discrete space
for $\umedio_h$ and  $\uu_h$ has been chosen,  the interpolation
error depending on $\emedio_h-{\bf e}_h^{m}$ in $(E_1)_h^{m+1}$ is
$\ee_i - {\bf e}_i^{m}$ and the interpolation error depending on
${\bf e}_h^{m+1}-\emedio_h$  in $(E_2)_h^{m+1}$ is zero. Finally,
since $ \Big(\gradH \cdot \langle I_h {\bf u}(t_{m+1}) \rangle, q_h
\Big)_S=0$, the corresponding interpolation error $ \Big(\gradH
\cdot \langle {\bf e}_i^{m+1} \rangle, q_h \Big)_S=0$, hence it is not appear in $(E_2)_h^{m+1}$.

 Therefore, adding  $(E_1)_h^{m+1}$
and  $(E_2)_h^{m+1}$, one arrives at:
$$
  \left \{ \begin{array}{l}
\dis\frac{1}{k}\Big(\ee_h -{\bf e}_h^m , {\bf v}_h \Big)+ \Big(
\nabla \ee_h ,\nabla  {\bf v}_h \Big) - \Big( e_{p,h}^{m+1} , \gradH
\cdot \langle{\bf v}_h \rangle \Big)_S
\\ \q
=\dis -\frac{1}{k}\Big(\ee_i - {\bf e}_i^{m} , {\bf v}_h \Big)+{\bf
NL}^{m+1} ({\bf v}_h )
  + \Big(\mathcal{E}^{m+1} , {\bf v}_h \Big), 
  \\
\Big(\gradH \cdot \langle\ee_h\rangle ,  q_h  \Big)_S=0  .
\end{array} \right.
\leqno{(E)_h^{m+1}}
$$
Due to the choice of the   interpolation operator $(I_h,J_h)$
related to the hydrostatic Stokes problem,  the interpolation error
$ \Big(\nabla \ee_i , \nabla {\bf v}_h\Big) + \Big(e_{p,i}^{m+1} ,
\gradH \cdot \langle {\bf v}_h \rangle \Big)_S =0$, hence it does not
appear in $(E)_h^{m+1}$.

\subsection{Estimates for the vertical velocity}

 From $ (E_0)_h^{m}$, the following estimate holds:
\begin{equation}\label{est-errorvertical}
| \partial _z e_{3,h}^m |  \le C ( | \gradH \cdot {\bf e}_h^m | + |
\gradH \cdot {\bf e}_i^m | ).
\end{equation}
Therefore, by using inequality (\ref{du_4}),
\begin{equation}\label{est2-errorvertical}
\| e_{3,h}^m \|_{L_z^{\infty}L_{\bf x}^2}^2 \le C ( \| {\bf e}^m_h
\| + \| {\bf e}^m_i \|).
\end{equation}
 On the other hand, the approximation property (\ref{estHz}) and the regularity $u_3\in L^\infty(0,T;H^{l+1})$ imply:
\begin{equation}\label{est-interpo-vertical}
| \partial_z e_{3,i} ^m | \le C\, h^l\, \|  u_3 (t_m) \|
_{H^{l+1}}
\le C\, h^l.
\end{equation}
 Therefore, by using again (\ref{du_4}),
\begin{equation}\label{est-interpo-vertical-bis}
\| e_{3,i} ^m \|_{L_z^{\infty}L_{\bf x}^2} \le C\, h^l.
\end{equation}

\subsection{$O(\sqrt{k}+h^l)$  error estimates for both velocities
in $l ^{\infty}( {\bf L}^2) \cap l ^{2}( {\bf H}^1)$}

\begin{theorem}\label{dt1}
We assume {\bf (H0)}-{\bf (H3)}, {\bf (R1)} and  $|{\bf e}_h^0 | \le
C\, h^l $. Then, there exists $k_0>0$ such that for any $k\le k_0$,  the following error
estimates hold
\begin{equation}\label{dtr1}
\| {\bf e}_h^{m+1/2}  \| _{ l ^{\infty}( {\bf L}^2) \cap l ^{2}(
{\bf H}^1)}
 + \| {\bf e}_h^{m+1}  \| _{ l ^{\infty}( {\bf L}^2) \cap
l ^{2}( {\bf H}^1)} \le C\, (\sqrt{k} + h^l ),
\end{equation}
\begin{equation}\label{dtr1bis}
\| {\bf e}_h^{m+1/2} -{\bf e}_h^m \| _{ l ^{2}( {\bf L}^2)}
 + \| {\bf e}_h^{m+1} - \emedio_h \| _{ l ^{2}( {\bf L}^2)} \le C\,
 \sqrt{k} (\sqrt{k} + h^{l} ) .
\end{equation}
\end{theorem}
 Notice that in this result, the constraint $({\bf H})$ on parameters $(k,h)$
 is not necessary although $k$ small enough must be imposed.

\begin{proof} The main idea is  to make
 $2\, k\displaystyle\sum_{m=0}^{M-1} \left \{  \Big((E_1)_h^{m+1},
\emedio_h \Big) +\Big((E_2)_h^{m+1},{\bf e}_h^{m+1} \Big) \right
\}$.

In fact, making $2\, k\,  \Big((E_1)_h^{m+1},  \emedio_h \Big)$ and
taking into account that $ c\Big( {\bf U}_h^m , \emedio_h, \emedio_h
\Big)=0$, we arrive at
\begin{eqnarray}\nonumber
   &&  |\emedio_h |^2 -|{\bf e}_h^m|^2 + |\emedio_h - {\bf e}_h^m|^2+
2\, k\, \| \emedio_h  \|^2 \\ \label{des-interm}
   &=&  2\, k\  \Big( p_s(t_{m+1}),\gradH \cdot  \emedio_h \Big)
+ 2\, k\, c\Big({\bf E}_h^m+{\bf E}_i^m, {\bf u}(t_{m+1}),
\emedio_h \Big)
 \\ \nonumber
   &+ &   2\,k\, c\Big({\bf U}_h^m,\ee _i, \emedio_h \Big)+ 2\, k\,
\Big(\mathcal{E}^{m+1}-\delta_t\ee_i,\emedio_h \Big ) -\Big (\nabla \ee_i, \nabla\emedio_h
\Big) :=\displaystyle \sum_{i=1}^5 I_1
\end{eqnarray}

For brevity, we only bound  the more difficult terms  of the RHS of (\ref{des-interm}) (for more details, see \cite{g-re-space}, where these type of bounds have been made for the Navier Stokes case). We bound term $I_1$, by using that $  \Big( J_h p_s (t_{m+1}),\gradH \cdot  {\bf e}_h^m  \Big)=0$, as
\begin{eqnarray*}
I_1 &=& 
2\, k\,\Big ( p_s (t_{m+1}), \gradH \cdot  (\emedio_h - {\bf e}_h^m)
\Big ) + 2\, k\, \Big(e_{p,i}^{m+1} , \gradH \cdot  {\bf e}_h^m
\Big)
\\
&=& - 2\, k\, \Big(\gradH p_s(t_{m+1}),
  \emedio_h - {\bf e}_h^m \Big) +2\,k\, \Big(e_{p,i}^{m+1} , \gradH \cdot
   {\bf e}_h ^m  \Big) \\
&\le & \varepsilon  | \emedio_h -{\bf e}_h^m |^2 + C\,k^2 \|
p_s(t_{m+1})\| ^2 + \varepsilon \, k \| {\bf e}_h^m \| ^2
+ C\, k\, h^{2l} \| p_s (t_{m+1}) \|_{H^l(S)} ^2
\end{eqnarray*}
The interpolation part of $I_4$ is 
$$
  2 \, k\, \Big(\delta_t\ee_i  , \emedio_h \Big) = 2\,  k\,
  \Big({\bf e}_i (\delta_t{\bf u}(t_{m+1})), \emedio_h \Big)
\le  \varepsilon \, k\, \|\emedio_h \|^2+ C\,  h^{2l}
\int_{t_m}^{t_{m+1}}
  \| {\bf u}_t \|_{H^l}^2 
$$
The discrete vertical part of $I_2$ is 
$$
  2\,k\, c\Big(e_{3,h}^m , {\bf u}(t_{m+1}), \emedio_h \Big)
  = 2\, k\, \Big(e_{3,h}^m \,\partial_z {\bf u}(t_{m+1}),
  \emedio_h \Big)+
k\,\Big(\partial_z e_{3,h}^m \, {\bf u}(t_{m+1}), \emedio_h
\Big):=L_1+L_2 .
$$
By using that ${\bf u} \in L^{\infty} ({\bf H}^{l+1} )$
and (\ref{est2-errorvertical}),
\begin{eqnarray*}
 L_1 &\le& 2\,k\, \| e_{3,h}^m \|_{L_z^{\infty}  L^2_{\bf x}}
 \|\partial_z {\bf u}(t_{m+1})\|_{L_z^{2} L^4_{\bf x}}
 \| \emedio_h \|_{L_z^{2}  L^4_{\bf x}}
 \\ &\le& C\,k\, (\| {\bf e}_h^m \| + \| {\bf e}_i^m \|)
  | \emedio_h |^{1/2} \| \emedio_h \| ^{1/2} \\
 & \le & \varepsilon \, k\, \| {\bf e}_h^m \| ^2 + C\,k\,h^{2l}
  +  \varepsilon \, k\, \| \emedio_h
  \|^2 +  C\,k \, | \emedio_h | ^2
\end{eqnarray*}
By using  that ${\bf u} \in L^{\infty} ({\bf
  H}^2 )$ and (\ref{est-errorvertical}),
\begin{eqnarray*}
L_2  & \le & k\,|\partial_z e_{3,h}^m | \, \|{\bf u}(t_{m+1})
\|_{L^{\infty}} \, | \emedio_h |  \le \varepsilon \,k\,  (\| {\bf
  e}_h^m \| ^2 + \| {\bf e}^m_i  \| ^2)
+ C\, k\, | \emedio_h |^2 \\
& \le & \varepsilon \,k\,\| {\bf e}_h^m \| ^2 +  C\, k\,  |
\emedio_h |^2 +  C\, k\,h^{2l} .
   \end{eqnarray*}
By a similar way, by using that ${\bf u} \in L^{\infty} ({\bf
  H}^2 )$, $u_3 \in L^{\infty} ({\bf  H}^{l+1} )$ and (\ref{est-interpo-vertical}), the vertical  interpolation part of $I_2$ is bounded as:
\begin{eqnarray*}
   2\,k\, c\Big( e_{3,i}^m , {\bf u}(t_{m+1}), \emedio_h \Big) &=& 2\,k \Big(
  e_{3,i}^m  \,\partial_z {\bf u}(t_{m+1}) , \emedio_h \Big)
  + k\,\Big(\partial_z  e_{3,i}^m \, {\bf u}(t_{m+1}) , \emedio_h \Big)
  \\ &\le&
C \,k\, \| e_{3,i} ^m \| _{H(\partial_z)} \, \|{\bf u} (t_{m+1})\|
_{H^2}
 \| \emedio_h \|
 \le \varepsilon \, k\, \| \emedio_h \|^2 + C\,k\,  h^{2l} .
\end{eqnarray*}
Now, we decompose $I_3$ as
  $$I_3 =
 2\,k\, c\Big({\bf U}_h^m ,\ee _i, \emedio_h \Big)
 =2\,k\, c\Big({\bf E}^m ,\ee _i, \emedio_h \Big)
 -2\,k\, c\Big({\bf U}(t_m) ,\ee _i, \emedio_h \Big),
$$
and their more difficult terms can be bounded,  using (\ref{est-errorvertical}),
 (\ref{est2-errorvertical}) and the inverse inequalities
  $ \| \emedio_h\|_{L_z^2 \, L_{\bf x}^{\infty}} \le C\,  h^{-1} |\emedio_h | $
   and  $\| \emedio_h\|_{L^3} \le C\,  h^{-1/2} |\emedio_h |$, as follows:
\begin{eqnarray*}
  2\,k\, c\Big( e_{3,h}^m + e_{3,i}^m , \ee_i , \emedio_h \Big) &=&
   2\,k \Big( (e_{3,h}^m +e_{3,i}^m)\,\partial_z \ee_i ,  \emedio_h \Big)
   \\&+&
   k \Big((\partial_z   e_{3,h}^m+ \partial_z e_{3,i}^m) \,\ee_i , \emedio_h \Big) \\
& \le &  2\,k \| e_{3,h}^m +e_{3,i}^m \|_{L_z^{\infty} L_{\bf x}^2}
\|\partial_z \ee_i \|_{L_{z}^2L_{\bf x}^2}  \|\emedio_h
\|_{L_{z}^2L_{\bf x} ^{\infty}}\\
&+&k \|\partial_z \,  e_{3,h}^m + \partial_z e_{3,i}^m \|_{L^2}
\|\ee_i \|_{L^6} \|\emedio_h \|_{L^3}
\\ & \le & C\, k\, (\| {\bf e}_h^m \|+ \| {\bf e}^m_i
\|+ h^l )\|\ee_i\|\,(h^{-1}+h^{-1/2}) | \emedio_h |
\\ & \le & C\, k\, (\| {\bf e}_h^m \|+ h^l )h^l\,h^{-1} | \emedio_h |
\\ & \le &
   \varepsilon \, k\, \| {\bf e}_h^m \|^2+ C\,k\,h^{2l}
+  C\,k \, | \emedio_h | ^2
\end{eqnarray*}
and
$$
2\,k\, c\Big(u_3(t_m) ,\ee _i, \emedio_h \Big)\le \varepsilon\, k
\, \| \emedio_h\|^2+ C\, k\, h^{2l} .
$$

Therefore, applying previous estimates to (\ref{des-interm}) and
making $| \emedio_h |^2 \le 2(|{\bf e}_h^m|^2  + | \emedio_h -{\bf
e}_h^m| ^2)$, we get
\begin{equation}\label{des-interm-b}
\begin{array}{l}
  |\emedio_h |^2 -|{\bf e}_h^m|^2 + |\emedio_h - {\bf e}_h^m|^2+
2\, k\, \| \emedio_h  \|^2
\\ \nonumber
   \le  C\,k ( |{\bf e}_h^m|^2  + | \emedio_h -{\bf e}_h^m| ^2) +2\, k\,
     \Big(\mathcal{E}^{m+1} , \emedio_h
\Big) + \varepsilon \, k \| \emedio_h \|^2
\\ +
 \varepsilon \, k\, \| {\bf e}_h^m \|^2
 +C\,k^2 + C\,k\,h^{2l} .
\end{array}
\end{equation}

 On the other hand, making $ 2\, k\Big( (E_2)_h^{m+1},{\bf
e}_h^{m+1} \Big)$, we arrive at
\begin{equation}\label{paso2}
   |{\bf e}^{m+1}_h |^2 - |{\bf e}^{m+1/2}_h|^2 + |{\bf e}^{m+1}_h - {\bf
     e}^{m+1/2}_h|^2+  k\,\Big\{ \|{\bf e}^{m+1}_h   \|^2
-\|  {\bf e}^{m+1/2}_h \|^2 + \|{\bf e}^{m+1}_h  -   {\bf
e}^{m+1/2}_h \|^2 \Big\} =0.
\end{equation}

Adding (\ref{des-interm-b}) and (\ref{paso2}) from $m=0$ to $r$
(with any $r<M$) and choosing $\varepsilon$ and $k$ small enough,
\begin{eqnarray*}
&&|{\bf e}^{r+1}_h |^2 + \sum_{m=0}^r \left(|{\bf e}^{m+1}_h - {\bf
     e}^{m+1/2}_h|^2 +\fra{1}{2} |{\bf e}^{m+1/2}_h - {\bf
     e}^{m}_h|^2 \right) \\
     &&\quad + k\sum_{m=0}^r \left(
      \fra{1}{2} \|{\bf e}^{m+1}_h\|^2 +
       \|{\bf e}^{m+1}_h - {\bf
     e}^{m+1/2}_h \|^2 +  \fra{1}{2} \|{\bf e}^{m+1/2}_h \|^2
      \right)
      \\ &&
       \le C\,k \sum_{m=0}^r |{\bf e}^{m}_h|^2 + \varepsilon \, k\, \| {\bf e}_h^0 \| ^2   + C(k+h^{2l})
\end{eqnarray*}
Therefore, applying discrete  Gronwall's Lemma, we can get
(\ref{dtr1}) and (\ref{dtr1bis}).
 \end{proof}

\subsection{$ O(k+h^l)$  error estimates for $\ee_h$ in $l
^{\infty}( {\bf L}^2) \cap l ^{2}( {\bf H}^1)$.}

\begin{theorem}\label{dt2}
Under  hypotheses of Theorem \ref{dt1}, ${\bf (R2)}$ and ${\bf
(H)}$, the following error estimate holds
\begin{equation}\label{dtr2}
 \| {\bf e}_h^{m+1}  \| _{ l ^{\infty}( {\bf L}^2) \cap
l ^{2}( {\bf H}^1)} \le C\, (k + h^l ).
\end{equation}
\end{theorem}
Note that, from (\ref{est-errorvertical}) and (\ref{dtr2}), we
also have
$$ \| e_{3,h}^{m+1} \|_{l^2 (H(\partial_z))} \le C\, (k+h^l).$$

\begin{proof}
 The main idea is  to make $2\,
k\sum_{m=0}^{M-1} \Big((E)_h^{m+1}, \ee_h \Big).$

In fact, making $2\, k\,  \Big((E)_h^{m+1},  \ee_h \Big)$, the
pressure term vanish, and we arrive at
\begin{eqnarray*}
   &&  |\ee_h |^2 -|{\bf e}_h^m|^2 + |\ee_h - {\bf e}_h^m|^2+
2\, k\, \| \ee_h  \|^2 \\
   &=&
 -2 \,k\Big(\delta_t\ee_i , \ee_h \Big)
 -2\, k\, c\Big({\bf E}_h^m, {\bf u}(t_{m+1}), \ee_h \Big)
-2\,k\, c\Big({\bf E}_i^m , {\bf u}(t_{m+1}) ,   \ee_h \Big)\\
   &-&2\,k\,c\Big( {\bf
U}_h^m , \emedio_h,\ee_h \Big) - 2\,k\, c\Big({\bf U}_h^m ,
\ee_i, \ee_h \Big)+ 2\, k  \Big(\mathcal{E}^{m+1} , \ee_h
\Big):=\sum_{i=1}^6 I_i\end{eqnarray*}

We bound the $I_i$ terms of similar way as in Theorem \ref{dt1}:
$$
  I_1=- 2\,k \Big({\bf e}_i
  (\delta_t  {\bf u}(t_{m+1})), \ee_h \Big)
\le  \varepsilon \, k\, \|\ee_h \|^2+ C\, h^{2l}
\int_{t_m}^{t_{m+1}} \| {\bf  u}_t \|_{H^{l}} ^2 .
$$
The vertical part of $I_5$ is bounded as
\begin{eqnarray*}
  2\,k\, c\Big(e_{3,h}^m , \ee _i, \ee_h \Big) &\le&
   C\, \|e_{3,h}^m\|_{L^\infty_zL^2_{\bf x}}\| \ee _i\|_{H^1}\|
\ee_h\|_{L^2_zL^\infty_{\bf x}} \le C\, \|{\bf
e}_{h}^m\|\,h\,h^{-1}|\ee_h| \\
   &\le & \varepsilon \, k\, \|{\bf e}_{h}^m \|^2+ C\,k(|\ee_h-{\bf
e}_{h}^m|^2+|{\bf e}_{h}^m|^2).
\end{eqnarray*}
The term $I_4=c\Big( {\bf U}_h^m ,
\emedio_h,\ee_h \Big)\not=0$, but  using that $ c\Big( {\bf U}_h^m
, \ee_h,\ee_h \Big)=0$, this term $I_4$ is decomposed as
$$I_4=- 2\,k\,
c\Big( {\bf E}^m , \emedio_h - \ee_h,\ee_h \Big)+ 2\,k\, c\Big(
{\bf U}(t_m) , \emedio_h - \ee_h,\ee_h \Big):=J_1+J_2.$$ The  more
complicate terms of $J_1$ are the vertical parts:
$$
2\,k \Big( (e_{3,h}^m + e_{3,i}^m)\, (\emedio_h -\ee_h) , \partial_z \ee_h \Big) \\
 +2\,  k\Big (\partial_z (  e_{3,h}^m +  e_{3,i}^m)\, (\emedio_h -\ee_h) , \ee_h \Big):=
J_{1,1} + J_{1,2} .
$$
Since $J_{1,2}$ is easier  to bound  than $J_{1,1}$, we only bound
$J_{1,1}$:
\begin{eqnarray*}
J_{1,1} &\le& 2\, k \,\| e_{3,h}^m + e_{3,i}^m\|_{L_z^{\infty}
L_{\bf x}^2} \| \emedio_h -\ee_h \|_{L_z^2 L_{\bf x}^\infty}
\|\partial_z \ee_h \|_{L^2_zL^2_{\bf x}}
\\ &\le& C\,k\,(\| {\bf e}_{h}^m\|+\| {\bf e}_{i}^m
\| + h^l \,\|u_3(t_m) \|_{H^{l+1}} ) \, h^{-1} |\emedio_h -\ee_h | \, \| \ee_h \|\\
 &\le& C\, k \, (h^{-2} \, | {\bf e}_h^m |
+1 )\, | \emedio_h -\ee_h | \, \| \ee_h \|
\\ & \le& \varepsilon \ k \,  \| \ee _h \| ^2 + C\, k\, h^{-4}
|\emedio_h -\ee_h |^2 \, |{\bf e}_h ^m | ^2 + C\, k \, |\emedio_h
-\ee_h |^2 .
\end{eqnarray*}
 Here,
we have used (\ref{est2-errorvertical}) and
(\ref{est-interpo-vertical-bis}) and the inverse inequalities
$$ \|\emedio_h -\ee_h \|_{L^2_z \, L^\infty_{\bf x}} \le
C \, h^{-1} |\emedio_h -\ee_h |\quad \hbox{and} \quad \|{\bf e}_h
^m \| \le C \, h^{-1} |{\bf e}_h ^m |.$$ Moreover, $J_2$ can be
bounded as
$$
J_2\le \varepsilon \ k \,  \| \ee _h \| ^2 + C\, k\,|\emedio_h
-\ee_h |^2 .
$$

Finally, adding from $m=0$ to $r$ (with any $r<M$) and taking into
account the bound
$$ C\,k\, h^{-4} \, \sum_m  |\emedio_h -\ee_h | ^2 \le  C\, h^{-2}
(k+h^{2l} ) \le C, $$ (where estimate (\ref{dtr1bis}) of Theorem
\ref{dt1} and ${\bf (H)}$ have been used)
 we can apply the discrete  Gronwall's Lemma,
 obtaining the desired estimates.
 \end{proof}

\subsection{$O( \sqrt{k}+h^l)$ error estimates for $\deltaee_h$  in $l ^{2}(
  {\bf L}^2) $  and for $ \ee_h $ in $l ^{\infty}({\bf H}^1) $ }

 We will use the following notations
 for the discrete derivative of errors
$$\delta _t  {\bf e}_h^{m+1}:=\displaystyle \frac{{\bf e}_h^{m+1}-{\bf
    e}_h^{m}}{k}, \qquad
 \delta _t {\bf e}_h^{m+1/2}:=\displaystyle \frac{{\bf e}_h^{m+1/2}-{\bf
     e}_h^{m-1/2}}{k}.$$

\begin{theorem}\label{eqdt1}
Assume  hypotheses of Theorem \ref{dt1}, $({\bf R3})$, $({\bf H})$  and $\|  {\bf e}_h ^0 \|\le
C\, h^l$. 
Then, the following error estimate holds
$$
 \| {\bf e}_h^{m+1}  \| _{ l ^{\infty}( {\bf H}^1)} +
\| \delta_t \ee_h\| _{l^2({\bf L}^2)} \le C\, (\sqrt{k} + h^l ).
$$
\end{theorem}

\begin{proof}
Making  $2\, k \Big((E)_h^{m+1}, \delta_t \ee_h \Big)$ the
pressure term vanish, and we arrive at
\begin{equation}\label{eh-deltateh}
\begin{array}{lcl}
   & \|\ee_h \|^2 -\|{\bf e}_h^m\|^2 + \|\ee_h - {\bf e}_h^m \|^2+
2\, k\, | \delta_t \ee_h  |^2 = 
- 2 \,k\Big(\delta_t\ee_i, \delta_t \ee_h \Big)&\\
&
  -2\, k\, c\Big({\bf E}_h^m, {\bf u}(t_{m+1}), \delta_t \ee_h \Big)
-2\,k\, c\Big({\bf E}_i^m , {\bf u}(t_{m+1}) , \delta_t   \ee_h \Big)&\\
& -2\,k\,c\Big( {\bf
U}_h^m , \emedio_h,\delta_t \ee_h \Big)
    - 2\,k\, c\Big({\bf U}_h^m
, \emedio _i, \delta_t \ee_h \Big)+ 2\, k  \Big(\mathcal{E}^{m+1}
, \delta_t \ee_h \Big)&\\
&    :=\sum_{i=1}^6 I_i &
\end{array}
\end{equation}
We must bound  the $I_i$ terms:
$$
  I_1
  = -2 \, k\Big({\bf e}_i (\delta_t
  {\bf u}(t_{m+1})), \delta_t \ee_h \Big)
\le  \varepsilon \, k\, |\delta_t\ee_h |^2+ C\, h^{2l}
\int_{t_m}^{t_{m+1}} \| {\bf  u}_t \|_{H^l} ^2 .
$$
Now, the term $I_4=-2 \, k\,c\Big( {\bf U}_h^m ,
\emedio_h,\delta_t \ee_h \Big)$ does not vanish,
$$I_4=2 \, k\,
c\Big( {\bf E}^m , \emedio_h ,\delta_t \ee_h \Big)-2 \, k\,c\Big(
{\bf U}(t_m) , \emedio_h ,\delta_t \ee_h \Big):=J_1+J_2.$$ The
more complicate terms to bound are the vertical parts of $J_1$:
$$
 2\,k \Big( (e_{3,h}^m + e_{3,i}^m)\, \partial_z \emedio_h,
 \delta_t  \ee_h \Big) \\
 + k\Big (\partial_z (  e_{3,h}^m +  e_{3,i}^m) \,\emedio_h
 ,\delta_t \ee_h \Big):=
J_{1,1} + J_{1,2} .
$$
By using (\ref{est2-errorvertical}) and
(\ref{est-interpo-vertical-bis}) and the inverse inequality $
\|\partial_z \emedio_h  \|_{L^2_z \, L^\infty_{\bf x}} \le C \,
h^{-1} \|\emedio_h \|$,
\begin{eqnarray*}
J_{1,1} &\le& 2\, k \,\| e_{3,h}^m + e_{3,i}^m\|_{L_z^{\infty}
L_{\bf x}^2} \| \partial_z \emedio_h \|_{L_z^2 L_{\bf x}^\infty}
\|\delta_t \ee_h \|_{L^2_zL^2_{\bf x}}
\\ &\le& C\,k\,(\| {\bf e}_{h}^m\|+  h^l \|{\bf u}(t_m)\|_{H^{l+1}}
 + h^l \|u_3(t_m)\|_{H^{l+1}} ) \, h^{-1} \|\emedio_h  \| \, |
\delta_t \ee_h |
\\ & \le& \varepsilon \ k \,  |\delta_t  \ee _h | ^2 + C\, k\, h^{-2}
\|\emedio_h \|^2 \|{\bf e}_h ^m \| ^2 + C\, k \, \|\emedio_h \|^2 .
\end{eqnarray*}
 By using the inverse inequality $ \| \emedio_h \|_{
L^\infty} \le C \, h^{-1/2} \|\emedio_h \|$,
\begin{eqnarray*}
J_{1,2} &\le& k\, \|\partial_z (  e_{3,h}^m +  e_{3,i}^m)\|_{L^2}
\|\emedio_h\|_{L^{\infty}}
 \|\delta_t \ee_h\|_{L^2} \le k \,(\| {\bf e}_h^m \|  + h^l ) \, h^{-1/2}
\|\emedio_h \| \, | \delta_t \ee_h |
\\ & \le& \varepsilon \ k \,  |\delta_t  \ee _h | ^2 + C\, k\, h^{-1}
\|\emedio_h \|^2 \|{\bf e}_h ^m \| ^2  + C\, k \,  \|\emedio_h
\|^2 .
\end{eqnarray*}
The $J_2$-term is bounded as
$$
J_2\le C \, k\,\|\emedio_h \|\,|\delta_t \ee_h | \le \varepsilon \
k \,  |\delta_t  \ee _h | ^2  + C\, k \,  \|\emedio_h \|^2 .
$$
 On the other hand, the vertical part of $I_2$ and $I_3$ are bounded
 as
$$
I_2 \le C\, k \Big( \| {\bf e}_h ^m\| +h^{l} \Big) \| {\bf
u}(t_{m+1}) \|_{H^3} |\delta_t \ee_h | \le \varepsilon \, k \, |
\delta_t \ee_h |^2 + C\, k \| {\bf e}_h ^m \| ^2 + C\, k \,h^{2l} ,
$$
$$
I_3\le \| e_{3,i}^m \|_{H(\partial_z)} \,  \| {\bf u}(t_{m+1})
\|_{H^3} \,  |\delta_t \ee_h | \le \varepsilon \, k \, | \delta_t
\ee_h |^2  + C\, k \, h^{2l} .
$$
We write $I_5$ as $$ I_5=2\,k\, c\Big(  {\bf E}^m , \ee_i
, \delta_t \ee_h \Big)-2\,k\, c\Big(  {\bf U}(t_m) , \ee_i  ,
\delta_t \ee_h \Big)
$$
 and its vertical part as
$$
2\,k \Big( (e_{3,h}^m + e_{3,i}^m) \partial_z \ee_i
 , \delta_t  \ee_h \Big) -
  k\Big (\partial_z (  e_{3,h}^m +  e_{3,i}^m) \ee_i
  ,\delta_t \ee_h \Big):=
K_1 + K_2
$$
We bound both terms as
\begin{eqnarray*}
K_1 &\le& 2\, k \,\| e_{3,h}^m + e_{3,i}^m\|_{L_z^{\infty} L_{\bf
x}^2} \| \partial_z \ee_i \|_{L_z^2 L^2_{\bf x}} \|\delta_t
\ee_h \|_{L^2_zL^{\infty}_{\bf x}}
\\ &\le& C\,k\,(\| {\bf e}_{h}^m\|+\| {\bf e}_{i}^m
\| +h^l ) \, h^l\|{\bf u}(t_{m+1})\|_{H^{l+1}} \, h^{-1} \, | \delta_t \ee_h |\\
&\le & \varepsilon \,k\, | \delta_t \ee_h |^2 + C\, k  \, \| {\bf
e}_h^m\| ^2 + C\, k \, h^{2l} ,
\end{eqnarray*}
\begin{eqnarray*}
K_2 &\le&  k \,\| \partial_z(e_{3,h}^m + e_{3,i}^m)\|_{L^2} \, \|
\emedio_i \|_{L^6}\,  \|\delta_t \ee_h \|_{L^3}
\\ &\le& C\,k\,(\| {\bf e}_{h}^m\|+\| {\bf e}_{i}^m
\| +h^l ) \, h^l\|{\bf u}(t_{m+1})\|_{H^{l+1}} \, h^{-1/2} \, | \delta_t \ee_h |\\
&\le & \varepsilon \,k\, | \delta_t \ee_h |^2 + C\, k \, h \, \|
{\bf e}_h^m\| ^2 + C\, k \, h^{2l} .
\end{eqnarray*}

Finally, taking into account the above estimates and adding (\ref{eh-deltateh}) from
$m=0$ to $r$ (with any $r<M$), since
 from estimates of
Theorem \ref{dt1} and ${\bf (H)}$,
$$ k\, h^{-2} \, \sum_m  \|\emedio_h \| ^2 \le  C \,  h^{-2}
(k+h^{2l} ) \le C  ,\ \ (l\ge 1), $$
 we can
apply the discrete  Gronwall's Lemma
 obtaining the desired estimates.
\end{proof}
\subsection{$O( \sqrt{k}+h^l)$ error estimates for $ e_{p,h}^{m+1}$  in $l ^{2}( {\bf L}^2) $ }

\begin{corollary}\label{epl2l2}
Assuming hypotheses of Theorem \ref{eqdt1}, one has
$$\| e_{p,h}^{m+1} \| _{l^2(L^2)} \le C\,(\sqrt{k}+ h^l ) .$$
\end{corollary}
The proof is rather standard, starting from estimates of previous
Theorems and applying the hydrostatic \emph{Inf-Sup} condition
{\bf (H1)}.

It is
important to remark that, up this moment,  the order obtained 
is $O(\sqrt{k}+ h^l)=O(h+h^l)$ under the constraint $k\le h^2$, then this order is optimal for
$O(h)$ approximation (i.e.~$l=1$). In the next Section, we study an
argument to arrive at optimal order $O(k+ h^l)$ for the case $l=2$.

\subsection{An alternative way for $O(h^2)$ accuracy ($l= 2$)}

\subsubsection{$O( \sqrt{k}+h^2)$ error estimates for $\deltaee_h$
and $\deltaemedio_h $ in $l ^{\infty}( {\bf L}^2) \cap l^{2}( {\bf
H}^1)$.}

Making  $\delta_t (E_1)^{m+1}_h$ and $\delta_t (E_2)^{m+1}_h $ for
each $m\ge 1 $, one obtains  $ \forall\, {\bf v}_h\in {\bf X}_h$:
$$
 \left\{  \begin{array}{l}
\dis\frac{1}{k}  \Big(\deltaemedio_h -\delta_t {\bf e}_h^{m}, {\bf
v}_h \Big) +  \Big  (\nabla   \deltaemedio_h ,\nabla\, {\bf v}_h
\Big ) - \Big( \delta_t p_s(t_{m+1}), \gradH \cdot \langle{\bf
v}_h\rangle \Big)_S \\\displaystyle
  = \Big(\delta_t \mathcal{E}^{m+1}, {\bf
v}_h \Big) +\delta_t {\bf NL}_h^{m+1} ( {\bf v}_h ) -\frac{1}{k}
\Big((\delta_t \ee_i - \delta_t {\bf e}_i^m), {\bf v}_h \Big) -
\Big(\nabla \, \delta_t \ee_i , \nabla \, {\bf v}_h \Big)
\end{array} \right.
\leqno{(D_1)^{m+1}_h}
$$
where
\begin{eqnarray*}
   \delta_t {\bf NL}_h^{m+1} ({\bf v}_h) &=&
   -c\Big( \delta_t {\bf E}^m  ,  {\bf u}(t_{m+1}), {\bf v}_h  \Big)
    - c\Big(\delta _t {\bf U}_h^m ,\emedio , {\bf v}_h  \Big) \\
   && - c\Big({\bf E}^{m-1} , \delta_t {\bf u}(t_{m+1}) , {\bf v}_h \Big)
    - c \Big( {\bf U}_h^{m-1} , \delta_t\emedio  , {\bf v}_h \Big )
\end{eqnarray*}
and, for all $({\bf v}_h,q_h) \in {\bf X}_h\times Q_h$,
$$
\left \{ \begin{array}{l} \displaystyle \frac{1}{k} \Big( \delta_t
{\bf e}_h^{m+1}-\delta_t \emedio_h,  {\bf v}_h  \Big)
  +   \Big( \nabla ( \deltaee_h -\deltaemedio_h), \nabla
 {\bf v}_h  \Big) = -   \Big(  \delta_t  p_{s,h}^{m+1},
\gradH\cdot \langle {\bf v}_h \rangle \Big)_S
\\
\Big( \gradH\cdot \langle \delta_t {\bf e}_h^{m+1} \rangle, q_h
\Big)_S =0.
\end{array} \right.
\leqno{(D_2)^{m+1}_h}
$$

Finally, adding $(D_1)^{m+1}_h$ and $(D_2)^{m+1}_h$ we obtain, for
all $({\bf v}_h, q_h) \in {\bf X}_h \times Q_h $:
$$
 \left \{ \begin{array}{l}
\displaystyle \frac{1}{k} \Big( \delta_t \ee_h - \delta_t {\bf
e}_h^m , {\bf v}_h \Big) + \Big(\nabla \, \delta_t \ee_h , \nabla
\, {\bf v}_h \Big) + \Big(\delta_t e_{p,h}^{m+1} ,\gradH \cdot
\langle{\bf v}_h\rangle\Big)_S
\\ \qquad = \Big(\delta_t \mathcal{E}^{m+1},
{\bf v}_h \Big) +\delta_t {\bf NL}_h^{m+1} ( {\bf v}_h )
-\displaystyle \frac{1}{k} \Big(\delta_t \ee_i -\delta_t  {\bf
e}_i^{m}, {\bf v}_h  \Big)  \\
\Big( \gradH\cdot\langle\delta_t {\bf e}_h^{m+1}\rangle , q_h
\Big)_S=0.
\end{array} \right.
\leqno{(D_3)^{m+1}_h}
$$

\begin{theorem}\label{dtt1} Under hypotheses of Theorem
   \ref{dt2} for $l=2$ (i.e.~$O(h^2)$ FE approximation),  ${\bf (R4)}$ and
assuming the following hypothesis for the first step of the scheme
  $$| \delta_t {\bf e}_h^{1} |\le C\, ( \sqrt{k} +h^2) ,$$
 then there exists $k_1 >0$ such that for any $k\le k_1$, 
$$
\| \deltaee_h \|_{l^\infty({\bf L}^2)\cap l^2({\bf H}^1)} +\|
\deltaemedio _h \|_{l^\infty({\bf L}^2)\cap l^2({\bf H}^1)}  \le
C\, (\sqrt{k}+ h^2) ,
$$
$$
\| \deltaee_h - \deltaemedio _h\|_{ l^2({\bf L}^2)} +\|
\deltaemedio _h -\delta_t{\bf e}_h^m \|_{ l^2({\bf L}^2)}  \le C\,
\sqrt{k}\, (\sqrt{k}+ h^2) .
$$
\end{theorem}

\begin{proof} Since the initial estimate   $| \delta_t {\bf e}_h^{1}
|\le C\, ( \sqrt{k} +h^2) $ has been assumed, it  suffices to
prove  the  generic estimate for  $\deltaee_h$ and
$\deltaemedio_h$, for each $m\ge 1$.

 Taking  $2\, k\, \delta_t \emedio_h  \in {\bf X}_h$  as
test function in  $(D_1)^{m+1}_h$, one has
 \begin{equation}\label{dte1}
   \begin{array}{l}
     |  \delta_t \emedio_h | ^2 -
|  \delta_t  {\bf e}_h^m | ^2 + |  \delta_t \emedio_h -\delta_t
{\bf e}_h^m | ^2 + 2\,   \, k \| \delta_t  \emedio_h \| ^2
\\
  = -2\, k\, \Big ( \fra{1}{k} (\deltaee_i -\delta_t  {\bf e}_i ^m
  )  ,  \delta_t \emedio_h  \Big )
-2\, k\,  \Big (\nabla \, \deltaee _i ,\nabla   \delta_t \emedio_h
\Big )
\\
 + 2\,k \Big( \delta_t p_s (t_{m+1}) ,\gradH \cdot \langle  \delta_t
 \emedio_h \rangle \Big)+2\, k \Big(\delta_t \mathcal{E}^{m+1}, \delta_t \emedio_h \Big)+
  2\,k \delta_t  {\bf NL}_h^{m+1}  (   \delta_t \emedio_h  )
  \\
:= I_1 +  I_2 + I_3 +I_4+ I_5.
 \end{array}
\end{equation}

We bound the RHS  of (\ref{dte1}) as in Theorem \ref{dt1}
(recalling that now one has $O(h^2)$ approximation)
$$
  I_1
\le \varepsilon \, k\, \|\delta_t\emedio_h \|^2+ C\,  h^4
\int_{t_m}^{t_{m+2}} \| {\bf u}_{tt} \|_{H^2}^2
$$
$$I_2
\le \varepsilon \, k\, \|\delta_t\emedio_h \|^2+ C\,h^4
\int_{t_m}^{t_{m+1}} \|{\bf u}_t \|_{H^3} ^2
 $$
$$
 I_3
\le \varepsilon  | \delta_t \emedio_h -\delta_t {\bf e}_h^m |^2 +
C\,k^2 \|\delta_t p_s(t_{m+1})\|_{H^1(S)}^2 +
 \varepsilon \, k \| \delta_t{\bf e}_h^m \| ^2 + C\, k\, h^4 \|
\delta_t p_s(t_{m+1}) \|_{H^2} ^2
$$
The bound of $I_4$ depending on the consistency error $\delta_t \mathcal{E}^{m+1}$ is not
problematic.

 Now, we bound the more complicate terms of $I_5$,
again as in the proof of Theorem \ref{dt1}:
 \begin{eqnarray*}
 && 2\,k\, c\Big(\delta_t  e_{3,h}^m  , {\bf u}(t_{m+1}),
   \delta_t\emedio_h \Big) \le
  \varepsilon \, k(\| \delta_t {\bf e}_h^m \| ^2 +\| \delta_t {\bf e}_i^m \| ^2 )
  +  \varepsilon \, k\,  \| \delta_t \emedio_h  \|^2 +
   C\,k  \, | \delta_t \emedio_h  | ^2\\
&&\le  \varepsilon \, k\, \| \delta_t {\bf e}_h^m \| ^2+ C\, h^4
\int_{t_m}^{t_{m+1}} \| {\bf u}_t \|_{H^3}^2 +  \varepsilon \, k\,
  \| \delta_t \emedio_h  \|^2 +  C\,k \, | \delta_t \emedio_h |^2 ,
\end{eqnarray*}
\begin{eqnarray*}
  2\,k\, c\Big(\delta_t e_{3,i}^m , {\bf u}(t_{m+1}), \delta_t\emedio_h \Big) &\le&
    \varepsilon \, k\, \| \delta_t\emedio_h \|^2  +
    C\, k\, \| \delta_t e_{3,i}^m \|_{H(\partial_z)}^2\\
 &\le&   \varepsilon \, k\, \| \delta_t\emedio_h \|^2  +
    C\, h^4 \int_{t_m}^{t_{m+1}} \|  \partial_t u_3   \|_{H^3}^2 ,
\end{eqnarray*}
\begin{eqnarray*}
 && 2\,k\, c\Big(\delta_t e_{3,h}^m + \delta_t e_{3,i}^m , \ee_i ,
  \delta_t\emedio_h\Big) \le
 \varepsilon \, k\, \Big(\| \delta_t {\bf e}_h ^m \|^2 + \| \delta_t {\bf
   e}_i ^m \|^2 + \|\delta_t e_{3,i}^m \|_{H(\partial_z)}^2 \Big)
 +  C\,k \, |\delta_t \emedio_h | ^2 \\
&&\le  \varepsilon \, k\, \| \delta_t {\bf e}_h ^m \|^2 + C\, h^ 4
\int_{t_m}^{t_{m+1}} (\| {\bf u}_t \|_{H^3}^2+ \|  \partial_t u_3
\|_{H^3}^2) +  C\,k \,  | \delta_t \emedio_h | ^2 .
 \end{eqnarray*}

The vertical part of $2\, k\, c\Big({\bf E}^{m-1} , \delta_t {\bf u}(t_{m+1}) , \delta_t \emedio_h\Big)$ is decomposed as follows:
$$  2\, k\,  \Big (
e_{3,h}^{m-1} + e_{3,i}^{m-1}, \partial_z \delta_t {\bf u}(t_{m+1}) ,
\delta_t \ee_h \Big ) + k\,  \Big ( \partial_z (e_{3,h}^{m-1} + e_{3,i}^{m-1}),
\delta_t{\bf u}(t_{m+1}), \delta_t \ee_h \Big )
:=L_1 +L_2 $$
Since $L_2$ is easier to bound than $L_1$, we only bound $L_1$:
$$L_1 \le \| e_{3,h}^{m-1} + e_{3,i}^{m-1} \|_{L_z^{\infty} L_{\bf
    x}^2}  \, \|  \partial_z \delta_t {\bf u}(t_{m+1})\|_{L_z^{2} L_{\bf
    x}^4}\, \| \delta_t
\emedio_h \|_{L_z^{2} L_{\bf
    x}^4}$$
$$ \le \varepsilon\, k\, \| {\bf e}_h^m \| ^2 + C\, k\, h^{2l}
+ \varepsilon \, k\, \| \delta_t \emedio_h \|^2 + C\, k\, |  \delta_t \emedio_h |^2 .
$$

On the other hand, we bound the other terms of $I_5$ which have
not similar terms in the proof of Theorem \ref{dt1}:
 $$c\Big( \delta_t{\bf  U}_h^m , \emedio_h,\delta_t \emedio_h \Big)=
 -c\Big( \delta_t{\bf
  E}^m , \emedio_h,\delta_t \emedio_h \Big)+c\Big( \delta_t  {\bf
  U}(t_m) , \emedio_h,\delta_t \emedio_h \Big).$$
The second term of the RHS is bounded by $ \varepsilon \, k\, \|
\delta_t {\bf e}_h^{m+1/2} \|^2 + C\,k\,|{\bf e}_h^{m+1/2}|^2 $.
With respect to the first term on the RHS, the more complicate
term to bound is the vertical part:
\begin{eqnarray*}
  &&  -2\, k\, c\Big( \delta_t e_3^m,\emedio_h , \delta_t  \emedio_h \Big)\\
   && = -2\, k
\Big( (\delta_t e_{3,h}^m + \delta_t e_{3,i}^m)\, \emedio_h ,
\partial_z \delta_t \emedio_h \Big) +  k \Big(
\partial_z (\delta_t e_{3,h}^m + \delta_t e_{3,i}^m)\, \emedio_h ,
\delta_t \emedio_h\Big):= J_1+J_2
\end{eqnarray*}

Since $J_2$ is easier to bound than $J_1$, we only bound $J_1$ (by using the inverse inequalities $ \| {\bf v}_h \|
_{H_{\bf x} ^1} \le C\, h^{-1} \, \| {\bf v}_h \|_{L_{\bf x}^2} $
and $ \| {\bf v}_h \| _{L_{\bf x} ^{\infty}} \le C\, h^{-1} \, \|
{\bf v}_h \|_{L_{\bf x}^2} $):
\begin{eqnarray*}
  J_1 &\le&  2\, k  \| \delta_t e_{3,h}^m + \delta_t e_{3,i}^m\|
_{L_z^{\infty} L_{\bf x}^2} \, \| \emedio_h \|_{L_z^{2}  L_{\bf
x}^{\infty}} \, \| \partial_z \, \delta_t \emedio_h
\|_{L^2_zL^2_{\bf x}}  \\
  &\le&  C\, k \Big( \|
\delta_t e_{h}^m \| + \| \delta_t e_{i}^m\| +
\frac{h^2}{k}\int_{t_{m-1}}^{t_m} \| \partial_t  u_3 \|  _{H^3}
\Big)\frac{1}{h}
\, |\emedio_h | \, \| \delta_t \emedio_h \| \\
   &\le&  C\, k\,\Big(\frac{1}{h} | \delta_t {\bf e}_h ^m |
    + \frac{h^2}{k}\int_{t_{m-1}}^{t_m}
    (\| {\bf u}_t \|_{H^3}+ \| \partial_t  u_3 \|  _{H^3}) \Big) \frac{1}{h}  \,
|\emedio_h | \, \| \delta_t \emedio_h \|  \\
   &\le& \varepsilon \, k\,   \|\delta_t \emedio_h \|^2 + C\, k\,
\frac{1}{h^4} \Big ( |\emedio_h -{\bf e}_h^m |^2  + | {\bf e}_h^m
|^2 \Big ) |\delta_t {\bf e}_h^m | ^2 \\
&+& C \,h^2\left(\int_{t_{m-1}}^{t_m} \| {\bf U}_t \|^2_{H^3}
\right) \Big ( |\emedio_h -{\bf e}_h^m |^2  + | {\bf e}_h^m |^2
\Big ) .
\end{eqnarray*}

Notice that to  apply   the discrete  Gronwall's Lemma with  the term $ C\, k\,
\frac{1}{h^4} \, | {\bf e}_h^m
|^2 \, |\delta_t {\bf e}_h^m | ^2$, it  is necessary  that
 $k\,\frac{1}{h^4} \,\sum _{m} | {\bf e}_h^m |^2 \le C$ and  this is true for $l=2$.
 In Section 6 below, we will see a modified scheme where it is possible  to obtain optimal error  estimates $ O(k+h^l)$ for $l=1$. 

On the other hand, we have to bound
\begin{eqnarray*}
c\Big( {\bf U}_h^{m-1} , \delta_t \emedio,\delta_t \emedio_h\Big)
&=& c\Big( {\bf U}_h^{m-1} , \delta_t \ee_i ,\delta_t \emedio_h\Big) \\
  &=&  -c\Big({\bf E}^{m-1}, \delta_t \ee_i ,\delta_t
\emedio_h\Big)+ c\Big( {\bf U}(t_{m-1}) ,\delta_t \ee_i ,\delta_t
\emedio_h\Big)
\end{eqnarray*}
(here we have used that $c\Big( {\bf U}_h^{m-1} , \delta_t \ee_h
,\delta_t \emedio_h\Big)=0$).
 The  more complicate term is the vertical part:
\begin{eqnarray*}
  &&  2\, k\, c\Big(  e_3^{m-1},\delta_t \ee_i , \delta_t \ee_i \Big) \\
   && = 2\, k \Big(
 (e_{3,h}^{m-1} +  e_{3,i}^{m-1})\, \delta_t\ee_i , \partial_z
\,\delta_t \emedio_h \Big) - k \Big(
\partial_z ( e_{3,h}^{m-1} + e_{3,i}^{m-1})\, \delta_t\ee_i ,
\delta_t \emedio_h\Big):= K_1+K_2
\end{eqnarray*}
Since $K_2$ is easier to bound than $K_1$, we only bound $K_1$:
\begin{eqnarray*}
K_1 &\le& 2\, k \,\| e_{3,h}^{m-1} + e_{3,i}^{m-1}\|_{L_z^{\infty}
L_{\bf x}^2} \|\delta_t \ee_i \|_{L_z^2 L_{\bf x}^2}
\|\partial_z\delta_t \ee_h \|_{L^2_zL^{\infty}_{\bf x}}
\\ &\le& C\,k\,(\| {\bf e}_{h}^{m-1}\|+ h^2 \,\| {\bf U}(t_{m-1})
\|_{H^3} ) \, h^2 k^{-1}(\int_{t_m}^{t_{m+1}}\|{\bf u}_t(t_{m+1}) \|_{H^2}) \, h^{-1} \, \| \delta_t\ee_h \|\\
 & \le& \varepsilon \ k \,  \| \delta_t\ee _h \| ^2 + C\, h^2
\int_{t_m}^{t_{m+1}}\|{\bf u}_t(t_{m+1}) \|^2_{H^2} ( \|{\bf e}_h
^{m-1} \|^2 + h^2).
\end{eqnarray*}

On the other hand, making $ 2\, k\,  \Big((D_2)_h^{m},\delta_t
{\bf e}_h^{m+1}   \Big)$, we arrive at
 \begin{equation}\label{dte2}
   \begin{array}{l}
   |\delta_t \ee_h |^2 - | \delta_t {\bf e}_h ^{m+1/2} |^2 +
   |\delta_t {\bf e}_h^{m+1}  - \delta_t {\bf e}_h ^{m+1/2} |^2\\
   +  k\,\Big\{ \|\delta_t {\bf e}_h^{m+1}  \|^2
-\|\delta_t {\bf e}_h ^{m+1/2}   \|^2 + \| \delta_t {\bf
e}_h^{m+1} - \delta_t {\bf e}_h ^{m+1/2} \|^2 \Big\}
   =0.
\end{array}
\end{equation}

 Reasoning as in Theorem \ref{dt1}, adding (\ref{dte1}) and
 (\ref{dte2}) from $m=0$ to $r$ (with any $r<M$), taking into account
 the previous estimates and choising $\varepsilon$ and $k$ small enough,
  we can apply the discrete  Gronwall's Lemma obtaining the desired
  estimates.
\end{proof}

\subsubsection{$O( k+h^2)$ error estimates for $\deltaee_h$  in
$l ^{\infty}( {\bf L}^2) \cap l^{2}( {\bf H}^1)$
}

\begin{theorem}\label{dtt2} Under hypotheses of Theorem~\ref{dtt1} and
 {\bf (R5)},
 assuming the following hypothesis for the first step of the scheme
  $$| \delta_t {\bf e}_h^{1} |\le C\, (k+ h^2) ,$$
 then
$$
\| \deltaee_h \|_{l^\infty({\bf L}^2)\cap l^2({\bf H}^1)}   \le
C\, (k+ h^2)  .
$$
\end{theorem}
\begin{proof} The main idea is to make $2\, k\,\Big ((D_3)^{m+1}_h ,
\delta_t \ee_h \Big )$. Now,  the pressure term vanish but the
term $c\Big( {\bf U}_{h}^{m-1} , \delta_t \emedio_h , \delta_t
\ee_h \Big)\not=0$ and can be decomposed as:
\begin{eqnarray*}
&& 2\, k\, c\Big( {\bf U}_{h}^{m-1} , \delta_t \emedio_h ,
\delta_t \ee_h \Big) = 2\, k\, c\Big( {\bf U} _{h}^{m-1} ,
\delta_t \emedio_h - \delta_t \ee_h , \delta_t \ee_h \Big) \\
&&= -2\, k\, c\Big( {\bf E}^{m-1} , \delta_t \emedio_h -\delta_t
\ee_h ,\delta_t \ee_h \Big)+2\,k\, c\Big(  {\bf U}(t_{m-1}) ,
\delta_t \emedio_h - \delta_t \ee_h , \delta_t \ee_h \Big):=
I_1+I_2.
\end{eqnarray*}
  The  more complicate term in $I_1$ is the vertical part:
\begin{eqnarray*}
   && k\, \Big( (e_{3,h}^{m-1}+ e_{3,i}^{m-1})\,
    (\delta_t \emedio_h - \delta_t \ee_h)
,\partial _z  \delta_t \ee_h \Big)
\\
&& \le k \| e_{3,h}^{m-1}+ e_{3,i}^{m-1} \|_{L^{\infty}_z L^2_{\bf
x}} \|\delta_t \emedio_h - \delta_t \ee_h \|_{L^2_z
L^{\infty}_{\bf x}}  \| \partial _z \delta_t \ee_h
\|_{L^2_zL^2_{\bf x}}
\\
&& \le \frac{k}{h^2} |{\bf e}_h^{m-1} | \,|\delta_t \emedio_h -
\delta_t \ee_h | \, \| \delta_t \ee_h \|+ \frac{k}{h} \| {\bf
e}_i^{m-1} \|\, |\delta_t \emedio_h - \delta_t \ee_h | \, \|
\delta_t \ee_h \|
\\
&&+ C \frac{k}{h}\, \| e_{3,i}^{m-1}\|_{L^{\infty}_z L^2_{\bf x}}
\,|\delta_t \emedio_h - \delta_t \ee_h | \, \| \delta_t \ee_h \|
\\
   &&\le  \varepsilon \, k\, \| \delta_t \ee_h \|^2 + C\, \frac{k}{h^4}
|\delta_t \emedio_h - \delta_t \ee_h |^2 \ |{\bf e}_h^{m-1} |^2 \\
&&+ C\, k \, h^2 \, |\delta_t \emedio_h - \delta_t \ee_h |^2
\end{eqnarray*}
(here, we have used the inverse inequalities $ \|{\bf
v}\|_{L^{\infty}_{\bf x}} \le C\, h^{-1}\, | {\bf v} |_{L^2_{\bf
x}} $ and $ \|{\bf v}\| \le C\, h^{-1}\, | {\bf v} |$  and the
estimates
 (\ref{est2-errorvertical}) and (\ref{est-interpo-vertical-bis})).
Since $|{\bf e}_h^{m-1} |^2 \le C\, (k^2+h^4)$, adding the third
term of the RHS of the previous inequality,
$$C\, k\, h^{-4} \, \sum_m
|\delta_t\emedio_h -\delta_t \ee_h | ^2 |{\bf e}_h^{m-1} |^2  \le
C\,k\, h^{-4} (k+h^4)(k^2+h^4) \le C \, (k^2+h^4),$$ where  ${\bf
(H)}$ and estimates of Theorem~\ref{dtt1} have been used.

 The  more complicate term in $I_2$ is the vertical part:
$$
2\,k\, c\Big(  u_3(t_{m-1}) , \delta_t \emedio_h - \delta_t \ee_h
, \delta_t \ee_h \Big)\le\varepsilon \, k\, \| \delta_t \ee_h \|^2
+ C\, k  \, |\delta_t \emedio_h - \delta_t \ee_h |^2 .
$$
Adding from $m=0$ to $r$ (with any $r<M$), we can  apply the
discrete  Gronwall's Lemma obtaining  the desired estimates.
\end{proof} Again, from estimates of previous Theorems and applying
the hydrostatic \emph{Inf-Sup} condition {\bf (H1)}, we arrive at
the following optimal error estimate for the pressure.
 \begin{corollary}\label{epl2l2bis}
 Assuming hypotheses of Theorem \ref{dtt2}, one has
 $$\| e_{p,h}^{m+1} \| _{l^2(L^2)} \le C\,(k+ h^2 ) .$$
 \end{corollary}

 Notice that  the previous estimate for the pressure is not
 obtained in the  $l^\infty(L^2)$ norm,
  due to the convective term depending on the intermediate error
   ${\bf e}^{m+1/2}_h$ which has not optimal approximation
  in $l^\infty({\bf L}^2)$, only in $l^2({\bf L}^2)$.

\section{A modified scheme with integral 
 computation for the vertical velocity.}\label{Se:structured meshes}


In this section, we will approximate the problem  $(Q)$.
We consider the variational formulation of $(Q)$ satisfied for the
exact solution $({\bf u},p_s)$ at $t=t_{m+1}$:
$$
\left\{   \begin{array}{l} \displaystyle
\Big(\displaystyle\frac{1}{k}({\bf u}(t_{m+1})-{\bf u}(t_m)),{\bf
v}\Big) + c \Big({\bf U}(t_{m}),{\bf u}(t_{m+1}), {\bf v}\Big) +
\Big(\nabla {\bf u}(t_{m+1}), \nabla {\bf v}\Big)
\\
-\Big(p_s(t_{m+1}) ,\gradH \cdot \langle {\bf v} \rangle \Big)_S =
\Big\langle{\bf f}(t_{m+1}), {\bf v}\Big\rangle_\Omega
+\Big\langle{\bf g}_s(t_{m+1}), {\bf v}\Big\rangle_{\Gamma_s}
+\Big(\mathcal{E}^{m+1},{\bf v}\Big) 
\\
\Big(\gradH \cdot \langle {\bf u}(t_{m+1}) \rangle,q \Big)_S=0 ,
 \end{array}
\right.\leqno{(Q)_w^{m+1}}
$$
for each $({\bf v},q) \in \Big ( 
{\bf
W}_{b,l}^{1,3}\cap{\bf L}^\infty \Big ) \times L^2_0(S)  $,
where
$$ u_{3} (t_m; {\bf x},z)  = \int_z^0 \gradH\cdot  {\bf
 u}(t_m;{\bf x},s)\, ds.
 $$
\subsection{$O(k+h^{l+1})$ for $\ee_h$ in $l^{2}({\bf L}^{2})$}


 We change in the scheme the
 computation of the vertical velocity $u_{3,h}^m$ in $(S_0)_h^{m+1} $, replacing Sub-step~0 by the vertical integral computation
  $$u_{3,h}^m ({\bf x},z)  = \int_z^0 \gradH\cdot  {\bf
 u}_h^m({\bf x},s)\, ds.
 $$
 Accordingly, we will change the vertical interpolation operator 
as follows
$$
\overline{K}_h u_{3} ({\bf x},z)  = \int_z^0 \gradH \cdot I_h {\bf
 u}({\bf x},s)\, ds,
$$
hence  $e_{3,i}^m = \int_z^0 \nabla_{{\bf x}} \cdot {\bf e}_i^m $.
 This interpolation operator $\overline{K}_h$ conserves the  properties already used  for the interpolation
$K_h$ and consequently, properties (\ref{est-errorvertical})-(\ref{est-interpo-vertical-bis}) can be used
 in this context. 

\begin{theorem} \label{mejorL2}
Assuming hypotheses of Theorem~\ref{dt2}, $({\bf R3})$ 
  and $\|A_h ^{-1} {\bf e}_h^0 \| \le C\, h^{l+1} $
(recall that $A_h$ is the discrete hydrostatic Stokes operator used  in
(\ref{prob-hydrostatic})), then there exists $k_0>0$ such that for any $k\le k_0$,  the following error estimates hold
\begin{equation}\label{mejoraL2}
\| {\bf e}_h^{m+1} \| _{ l ^{2}( {\bf L}^2)} \le C (k+ h^{l+1}).
\end{equation}
\end{theorem}

\begin{proof}
Taking  ${\bf v}_h =A_h^{-1}  {\bf e}_h^{m+1} \in {\bf V}_h$ in
$(E)_h^{m+1}$,
 we obtain
 \begin{eqnarray*}
    &&   \| A_h^{-1} \,  {\bf e}_h^{m+1} \| ^2 -
 \| A_h^{-1} \,  {\bf e}_h^{m} \| ^2 +
 \| A_h^{-1} \,  {\bf e}_h^{m+1}  -
  A_h^{-1} \,  {\bf e}_h^{m} \| ^2
+2\,  k | {\bf e}_h^{m+1} | ^2    \\
    && \le 2\, k \,  {\bf NL}_h^{m+1} (A_h^{-1} \,  {\bf e}_h^{m+1} )-
    2\,  k \Big( \delta_t {\bf e}_i^{m+1} , A_h^{-1}
 {\bf e}_h^{m+1} \Big)+ 2\, k \Big(\mathcal{E}^{m+1},
 A_h^{-1} {\bf e}_h^{m+1} \Big)
 :=I_1 +I_2 +I_3.
 \end{eqnarray*}

We bound the  more complicate terms of $I_1$:
$$2\, k\, c\Big(u_{3,h}^m, \emedio_h , A_h ^{-1} \ee_h \Big) =
 -2\, k\, c\Big( e_{3,h}^m , \emedio_h
, A_h ^{-1} \ee_h \Big)
$$
$$
+ 2\, k\, c\Big( u_3 (t_m) ,  \emedio_h , A_h
^{-1} \ee_h \Big):=L_1+L_2
$$
By using the inverse inequalities $\|{	\bf e}_h\|_{L^3}\le
h^{-1/2} |{	\bf e}_h|$ in $3D$-domains and $\|{	\bf e}_h\|_{L^{\infty}_{{\bf x}}}
\le h^{-1} \|{	\bf e}_h\|_{L^2_{{\bf x}}}$ in $2D$-domains:
\begin{eqnarray*}
L_1& =& k\, \Big (\partial_z  e_{3,h}^m \, \emedio_h , A_h ^{-1}
\ee_h \Big )-2\, k\,\Big ( e_{3,h}^m \, \emedio_h
, \partial_z A_h ^{-1} \ee_h \Big ) \\
&\le& k\, \Big(\|\partial_z  e_{3,h}^m \|_{L^2} \|\emedio_h
\|_{L^3} \, \|A_h ^{-1} \ee_h \|_{L^6}+2\,  \|e_{3,h}^m
\|_{L^{\infty}_z\, L^2_{{\bf x}}} \, \| \emedio_h \|_{L^{2}_z\,
L^{\infty}_{{\bf x}}} \,
\|\partial_z  A_h^{-1} \ee_h \|_{L^{2}_z\, L^{2}_{{\bf x}}} \Big) \\
&\le& C\, k\, \Big(\| {\bf e}_h ^m\| + \|{\bf e}_i^m \| \Big)
(h^{-1/2}+h^{-1})\Big(| \ee_h -
\emedio_h |+|\ee_h | \Big)  \| A_h^{-1} \ee_h \|\\
&\le& \varepsilon \, k\, | \ee_h |^2 + \varepsilon \, k | \ee_h -
\emedio_h |^2 + C\frac{k}{h^2} \Big(\|{\bf e}_h ^m \| ^2 + h^{2l}
\Big) \| A_h^{-1} \ee_h \|^2,
\end{eqnarray*}
\begin{eqnarray*}
L_2 &\le&  C\, k\, \Big( \|u_3 (t_m) \|_{L^{\infty}}+\|
\partial_z u_3(t_m)\|_{L^3} \Big) \Big( |
\emedio_h -\ee_h | +|\ee_h | \Big) \|  A_h ^{-1} \ee_h  \|
\\
&\le& \varepsilon \, k  \Big(|\ee_h |^2 +  |\emedio_h -\ee_h |^2
\Big) +C\, k \,  \|  A_h ^{-1} \ee_h  \|^2 .
\end{eqnarray*}

The vertical part of $2\, k\, c\Big({\bf U}_{h}^m ,\ee_i , A_h ^{-1}
\ee_h \Big)$ is
$$ 2\, k\, c\Big(u_{3,h}^m ,\ee_i , A_h ^{-1}
\ee_h \Big)=2\, k\, c\Big( e_{3,h}^m + e_{3,i}^m , \ee_i ,  A_h
^{-1}\ee_h\Big) - 2\, k\, c\Big(u_3 (t_m), \ee_i ,  A_h
^{-1}\ee_h\Big)$$ and its more complicate term is
$$2\, k\, c\Big(e_{3,h}^m, \ee_i , A_h ^{-1} \ee_h \Big)=2\,  k\, \Big(e_{3,h}^m\,
  \ee_i , \partial_z A_h ^{-1}\ee_h \Big)+k\, \Big(\partial_z e_{3,h}^m\,
   \ee_i , A_h ^{-1}\ee_h \Big):= N_1 +N_2   $$ 
   We bound
\begin{eqnarray*}
N_1 &\le & k\, \|e_{3,h}^m \|_{L_z^{\infty} L^2_{{\bf x}}} \ \|
\ee_i \|_{L_z^{2} L^2_{{\bf x}}}     \, \|\partial_z A_h^{-1}
\ee_h
\|_{L^{2} L^{\infty}_{{\bf x}}} \le C\, k\, \| {\bf e}_h^m  \|   \, |\ee_i |\, h^{-1} \,  \|A_h ^{-1} \ee_h \|\\
&\le & C\, \frac{k}{h^2} \|{\bf e}_h ^m \| ^2 \, \| A_h ^{-1}
\ee_h \| ^2 + C \,k \, h^{2(l+1)} ,
\end{eqnarray*}
\begin{eqnarray*}
N_2 &\le & k\, \| \nabla_{\bf x} \cdot {\bf e}_{h}^m \|_{L^{3}} \|
\ee_i \|_{L^2} \| A_h^{-1} \ee_h
\|_{L^6} \le C\, k\,h^{-1/2}  \|  {\bf e}_h ^m \|  \, |\ee_i | \|A_h ^{-1} \ee_h \|\\
&\le & C\, \frac{k}{h} \|{\bf e}_h ^m \| ^2 \, \| A_h ^{-1} \ee_h
\| ^2 + C \, k \,  h^{2(l+1)} .
\end{eqnarray*}
Finally, by a similar way, we bound
\begin{eqnarray*}
 2\, k\,c\Big(e_{3,i}^m, \ee_i , A_h^{-1} \ee_h \Big)
& \le &C\, k\, \Big(\| {\bf e}_i ^m \| \, |\ee_i |\,h^{-1} \|  A_h
^{-1} \ee_h \| + h^{-1/2} \| {\bf e}_i ^m \| \,
|\ee_i | \, \|  A_h ^{-1} \ee_h \| \Big)\\
& \le & C\, k\,  |\ee_i |\, \|  A_h ^{-1} \ee_h \| \le C\, k \, \|
A_h ^{-1} \ee_h \| ^2 + C \, k \, h^{2(l+1)} .
\end{eqnarray*}

Other conflictive term of $I_1$ is:
$$2\, k\, c\Big(e_{3,h}^m, {\bf u} (t_{m+1}), A_h ^{-1} \ee_h \Big)
= 2\, k\, \Big( e_{3,h}^m \,  \partial_z {\bf u} (t_{m+1}) , A_h ^{-1}
\ee_h \Big)
+ k\, \Big( \partial_z e_{3,h}^m  \,   {\bf u}
(t_{m+1}) , A_h ^{-1} \ee_h \Big):=J_1+J_2 $$
 To bound $J_1$, it is necessary
the explicit expression of
 $e_{3,h}^m = \int _z^0  \nabla_{{\bf x}} \cdot {\bf e}_h ^m =
 \nabla_{{\bf x}} \cdot \int _z^0 {\bf e}_h ^m $
 which let us to integrate by parts,
\begin{eqnarray*}
J_1&= &k \, \Big ( \nabla _{{\bf x}} \cdot (\int_z^0 {\bf e}_h^m)
\,
\partial_z {\bf
  u}(t_{m+1}), A_h ^{-1} \ee_h  \Big )
  \\
  &=& - k\, \Big (\int_z^0 {\bf e}_h
^m ,\nabla_{\bf x} \partial_z {\bf u}(t_{m+1}) \, A_h^{-1} \ee_h +
\partial_z {\bf u} (t_{m+1}) \, \nabla_{{\bf x}} A_h^{-1} \ee_h
\Big )
\\
 &\le & k\,\| \int_z^0 {\bf e}_h ^m \|_{L^{\infty}_z L^2_{\bf x}} \Big(
 \|\nabla_{\bf x} \partial_z {\bf u}(t_{m+1}) \|_{L^{6/5}_z
    L^3_{\bf x}}  \| A_h^{-1} \ee_h\| _{L^6_zL^6_{\bf x}}+
    \| \partial_z {\bf u}  (t_{m+1})\|_{L^{2}_z L^{\infty}_{\bf x}}
     \|\nabla_{{\bf x}} A_h^{-1}\ee_h \|_{L^2_{z}L^2_{\bf x}} \Big) \\
 &\le & \varepsilon \, k\,|{\bf e}_h ^m|^2 +C\, k \, \|A_h^{-1} \ee_h \|^2 .
\end{eqnarray*}
In a similar way, by using that $\partial_z e_{3,h}^m=-\gradH\cdot
{\bf e}_h ^m $,
\begin{eqnarray*}
J_2
 &\le & k\, |  {\bf e}_h ^m | \Big(
\|\nabla_{\bf x}   {\bf u}(t_{m+1}) \|_{
    L^3} \, \| A_h^{-1} \ee_h\| _{L^6}+ \|  {\bf u}
  (t_{m+1})\|_{ L^{\infty}}  |\nabla_{{\bf x}} A_h^{-1}
\ee_h | \Big)\\
&\le & \varepsilon \, k\,|{\bf e}_h ^m|^2 +C\, k \, \|A_h^{-1}
\ee_h \|^2 .
\end{eqnarray*}

On the other hand, we bound the $I_2$-term as follows
\begin{eqnarray*}
  I_2  &=& -2\, k   \Big( \delta_t {\bf e}_i^{m+1} , A_h^{-1} \, {\bf
e}_h^{m+1} \Big)\le C \,k \|A_h^{-1} {\bf e}_h^{m+1} \| ^2
+ \varepsilon \,k |{\bf e}_i (\delta_t \uu) |^2  \\
   &\le & C \,k \|A_h^{-1} {\bf
  e}_h^{m+1} \| ^2 + \varepsilon \, k\, h^{2(l+1)}
   \| \delta_t \uu \| _{H^{l+1}}^2 \\
   &\le & C\, k\, \|
A_h^{-1} {\bf e}_h^{m+1} \| ^2 + \varepsilon \, h^{2(l+1)}\,
\int_{t_m}^{t_{m+1}} \| {\bf u}_t \| _{H^{l+1}}^2 .
\end{eqnarray*}

Finally, the $I_3$-term is easy to bound by using that $ {\bf u}_{tt} \in L^2({\bf L}^2)$ given by hypothesis $({\bf R3})$.

Adding from $m=0$ to $r$ (for any $r<M$), since one has 
$ \|{\bf e}_h ^m \| ^2\le C\, h^2 $
and $k \sum_m | {\bf
e}_h^{m+1/2}-{\bf e}_h ^m |^2\le C k(k+h^{2l})\le C (k^2 +
h^{2(l+1)})$   owing to ${\bf (H)}$,
 we can  apply the generalized discrete  Gronwall's Lemma,  
obtaining  the desired estimates for $k$ small enough.
\end{proof}
\subsection{$O(k+h^l)$ for $e_{p,h}^{m+1}$ in $l^2( L^2)$ }
Owing to the improved error estimate obtained in the above Subsection, now we
can prove the same error estimates obtained in Theorems \ref{dtt1} and
\ref{dtt2} also for $l=1$, that is,  using $O(h)$ FE-approximation.

Indeed, in the proof of Theorem \ref{dtt1},
 we can apply the discrete  Gronwall's Lemma, since the  term $ C\, k\,
\displaystyle h^{-4} \, \displaystyle \sum_{m=0}^r | {\bf
e}_h^m |^2 \, |\delta_t {\bf e}_h^m | ^2$ appear, and now for
$l=1$, owing to (\ref{mejoraL2}) and ${\bf (H)}$, we have
 $$
 C\,k\,\frac{1}{h^4} \,\sum_{m=0}^r | {\bf e}_h^m |^2 \le C.
 $$
 The rest of the proof is similar, arriving at the following result:
\begin{corollary}\label{epl1}
 Assuming hypotheses of Theorem  \ref{dtt2} and Theorem \ref{mejorL2},
 one has
 $$\| e_{p,h}^{m+1} \| _{l^2(L^2)} \le C\,(k+ h^l ) .$$
 \end{corollary}

\section{Approximation of the  Coriolis term}\label{coriolis}

  Looking at
  the  results obtained in previous Sections, we consider that the more convenient
  forms to introduce the Coriolis term in the scheme can be the following (the Coriolis term will be   always refereed at the
 end-of-step velocity, either  $\uu_h$ or ${\bf u}_h^m$, because it is the better approximation in  the scheme):
\begin{enumerate}
%

\item Considering in $(S_1)_h^{m+1}$ the explicit term ${\bf b}
  ({\bf u}_h^m)$. In the  proof of Lemma \ref{apriori_est}, this term introduces the extra-term 
$$\Big({\bf b}
  ({\bf u}_h^m), \umedio_h\Big)=\Big({\bf b}
  ({\bf u}_h^m), \umedio_h -{\bf u}_h^{m}
  \Big)$$  which  produces  an
artificial exponential bound in time in the stability
estimates. Indeed, bounding as
$$\Big({\bf b}
  ({\bf u}_h^m), \umedio_h -{\bf u}_h^{m}
  \Big) \le C\,  |
  {\bf u}_h^m |^2+ \varepsilon |\umedio_h -{\bf u}_h^{m}|^2$$
and applying the  discrete Gronwall inequality, a new exponential bound appears. With respect to the
error estimates, some new terms  in $(E_1)_h^{m+1}$ and in $(E_2)_h^{m+1}$ appear, although these terms do not add new difficulties.

\item Considering the following Coriolis correction strategy: to introduce in $(S_1)_h^{m+1}$ the explicit term ${\bf b} ({\bf u}_h^m)$ and in $(S_2)_h^{m+1}$ the correction term ${\bf b}
  ({\bf u}_h^{m+1} -{\bf u}_h^m)$.
This  correction scheme  works in a similar way to scheme given in point 1  with respect to the 
stability and error estimates.


\end{enumerate}

Notice that in the two previous cases, the computation of the two components of velocity $\umedio_h$ is decoupled. Finally, the implementation  of Hydrostatic Stokes step (Sub-step~2) is simpler in the case 1. 

\section{Conclusions}

In this paper we have developed two  new ways of handling  stable and convergent  approximations for the Primitive Equations, based on the reformulations (Q) and (R) of the problem, being (Q) an  integral-differential formulation and (R) a fully  differential one.

 In both cases, vertically structured meshes are needed, and the error estimates are deduced under the same constraint $k\le  h^2$.
 
 With respect to computational  implementation, scheme based on (R) is simpler than scheme related to (Q), although the latter satisfies  more analytical error estimates. In fact, (Q)-scheme satisfies optimal accuracy   $O(k+h^{l+1})$ in the $L^2(\Omega)$-norm  for the velocity and   $O(k+h^{l})$  in the $H^1(\Omega)\times L^2(\Omega)$-norm for the velocity and pressure, whereas for the  (R)-scheme   this optimal  accuracy  can be proved only  in the $H^1(\Omega)\times L^2(\Omega)$-norm for the velocity and pressure  when $l=2$, that is,  using $O(h^2)$ FE-approximation. 
 
  Nevertheless, the constraint $k\le h^2$, imposed in both schemes, is compatible with the accuracy $O(k + h^2)$ which is satisfied in the $H^1\times L^2$-norm for velocity and pressure when  $l=2$ and only in the $ L^2$-norm for the velocity when $l=1$. Therefore, by using $O(h^2)$ FE-approximation ($l=2$) is more convenient the (R)-scheme because is simpler to implement.

\section*{Appendix}
\noindent{\bf Proof} [of Lemma \ref{regw16}]: 
Let  $A^{-1}{\bf v}\in {\bf V} $ be  the solution of the
hydrostatic Stokes Problem  with second member ${\bf v}$. This solution verifies (see \cite{g-r} for the Stokes case)
\begin{equation}\label{estrella}
\|A^{-1} {\bf v}_h -A_h^{-1}{\bf v}_h \|\le C\, h \, |{\bf v}_h | .
\end{equation}

On the other hand, 
$$ \| A_h^{-1}{\bf v}_h \| _{W^{1,6}} \le   \| A_h ^{-1} {\bf
  v}_h -\widetilde{I}_h  A^{-1}{\bf v}_h  \|
  _{W^{1,6}}+ \|  \widetilde{I}_h\, A^{-1} {\bf
  v}_h   \| _{W^{1,6}} $$
 where $\widetilde{I}_h$ is
an interpolator with respect to ${\bf X}_h $.

We bound the RHS as follows, using  hypothesis {\bf
(H0)},  the stability property $\|\widetilde{I}_h {\bf
v}_h\|_{W^{1,6}}  \le C \| {\bf v}_h\|_{W^{1,6}} $ and the inverse inequality $\| {\bf v}_h \|_{W^{1,6}} \le C\,h^{-1}  \| {\bf v}_h \|$ :
$$\|   \widetilde{I}_h A^{-1} {\bf v}_h \| _{W^{1,6}} \le
C\, \| A^{-1}{\bf v}_h  \|_{W^{1,6}}  \le C \|A^{-1}{\bf v}_h
\|_{H^2} \le C\, |{\bf v}_h | ,
$$
$$\|   A_h ^{-1} {\bf
  v}_h -\widetilde{I}_h  A^{-1}{\bf v}_h \|_{W^{1,6}} \le \frac{C}{h}
 \|  A_h ^{-1} {\bf
  v}_h -\widetilde{I}_h  A^{-1}{\bf v}_h \| \le  \frac{C}{h}
\Big(\| A_h ^{-1} {\bf
  v}_h - A^{-1}{\bf v}_h     \| +
 \|A^{-1}{\bf v}_h  - \widetilde{I}_h  A^{-1}{\bf v}_h   \| \Big) .
$$

Finally, applying (\ref{estrella}) and the error interpolation
inequality 
$$ \|A^{-1}{\bf v}_h  - \widetilde{I}_h \, A^{-1}{\bf
v}_h \| \le C\, h\, \|A^{-1}{\bf v}_h   \|_{H^2} \le C\, h |{\bf
v}_h|,$$ we arrive at 
$\|   A_h ^{-1} {\bf
  v}_h -\widetilde{I}_h \, A^{-1}{\bf v}_h \|_{W^{1,6}}\le C\,|{\bf
  v}_h|.$ 
Therefore, we conclude
$$  \|  A_h ^{-1} {\bf v}_h  \|_{W^{1,6}} \le C\, | {\bf v}_h |.
$$

\renewcommand{\baselinestretch}{1.1}


\begin{thebibliography}{99}
\footnotesize

\bibitem{AG}{\sc P.~Az\'erad, F.~Guill\'en. }
{\sl \'Equations de Navier-Stokes en bassin peu profond: l'approximation
hydrostatique. } C.~R.~Acad.~Sci.~Paris,  S\'erie I {\bf 329} (1999),   961-966.

\bibitem{AG2}{\sc P.~Az\'erad, F.~Guill\'en. }
{\sl Mathematical justification  of the hydrostatic approximation
in the primitive equations of geophysical fluid dynamics } Siam J.
Math. Anal., {\bf 33} (4) (2001), 847-859.




\bibitem{bermejo} {\sc R. Bermejo.} {\sl Velocity Error
    Estimates for a Semi-Lagrangian Ocean General Circulation  Model.}
  Actas de las II Jornadas de Análisis de Variables y Simulación
  Numérica del Intercambio de Masas de Agua  a través del Estrecho de
  Gibraltar, Cádiz, (2000), 19-34.

\bibitem{bermejo2} {\sc R. Bermejo, P. Galán del Sastre.} {\sl
    Long-Term Behavior of the Wind Stress Circulation of a Numerical
    North Atlantic Ocean Circulation Model.} European Congress on
  Computational Methods in Applied Sciences and Engineering, ECCOMAS
  (2004), 1-21.



\bibitem{BL}{\sc O.~Besson, M.~R.~Laydi. }
{\sl Some Estimates for the Anisotropic Navier-Stokes Equations and for the
Hydrostatic Approximation. } M2AN-Mod.~Math.~Ana.~Num., {\bf 7} (1992) 855-865.


\bibitem{bch}{\sc J.~Blasco, R~Codina, A.~Huerta. }
{\sl A fractional-step method for the incompressible Navier-Stokes
equations related to a predictor-multicorrector algorithm. }
 Internat. J.~Numer.~Methods Fluids, {\bf 28} (10) (1998), 1391--1419.

 \bibitem{cao-titi} {\sc C. Cao, E.S. Titi.} {\it Global well-posedness
    of the three-dimensional viscous primitive equations of large
    scale ocean and atmosphere dynamics.} Annals of Mathematics,
{\bf 166}(1) (2007), 245-267.

\bibitem{cas1}{\sc V.~Casulli, R.T.~Cheng. }  {\sl Semi-implicit finite difference methods for three-dimensional sahllow water flow.} 
Internat. J. Numer. Methods Fluids, {\bf 15},  (1992) 629-648.

\bibitem{cas2}{\sc V.~Casulli, E.~Cattani. }  {\sl Stability, Accuracy and
 Efficiency of a semi-implicit method for tree-dimensional shallow
  water flow. } Comp.~Math.~Applic., {\bf 27}, 4 (1994) 99-112.
  
   \bibitem{cas4}{\sc V.~Casulli, R.A..~Walters. }  {\sl A robust finite element model for hydrostatic surface water flows.} Communications in Numerical Methods in Engineering, {\bf 14},  (1998) 931-940. 

%


\bibitem{CB}{\sc B.Cushman-Roisin, J.M. Beckers.}
 {\sl Introduction to Geophysical Fluid Dynamics - Physical and
  Numerical Aspects.} Academic Press, 2009.




\bibitem{cg}{\sc T.~Chac\'on, F.~Guill\'en. } {\sl An intrinsic analysis of
existence of solutions for the hydrostatic approximation of Navier-Stokes
equations.}  C.~R.~Acad.~Sci.~Paris, S\'erie I {\bf 329} (2000), 841-846.

\bibitem{c-rg-04}{\sc  T.~Chacón; D.~Rodríguez-Gómez. }{\sl  A stabilized space-time
discretization for the primitive equations in oceanography.}
Numer. Math., {\bf 98} (3) (2004),  427-475.

\bibitem{c-rg-05}{\sc T.~Chacón; D.~Rodríguez-Gómez.} {\sl A numerical solver for the
primitive equations of the ocean using term-by-term
stabilization.} Appl. Numer. Math., {\bf 55} (1) (2005),  1-31.


\bibitem{cgs}{\sc T.~Chac\'on Rebollo, M.~Gómez Mármol, I.~Sánchez Muñoz. } {\sl Numerical solution of the Primitive Equations of the ocean by the Orthogonal Sub-Scales VMS method.}  Appl. Numer. Math. {\bf 62} (2012) 342-359.


\bibitem{g-r} {\sc V.~Girault, P.A.~Raviart. }
{\sl Finite Element Methods for Navier-Stokes Equations. }
 Springer-Verlag, 1986.

\bibitem{handbook-Glowinski} R. Glowinski,  Numerical  Methods for Fluids (Part 3), Handbook of Numerical Analysis, vol. IX, P.~G. Ciarlet
and J.~L. Lions, eds., North-Holland, 2003.

\bibitem{gmma}{\sc  F.~Guill\'en-Gonz\'alez, N.~Masmoudi, M.A.~Rodr\'{\i}guez-Bellido. }
{\sl Anisotropic Estimates and strong solutions of the Primitive
Equations } J.~Diff.~Integral Equations, {\bf 14} (11) (2001),
1381-1408.


 \bibitem{sal}{\sc F.~Guillén-González, M.~V.~Redondo-Neble,
     J.~R.~Rodríguez-Galván}
 {\sl Análisis Numérico y resolución efectiva de las Ecuaciones
   Primitivas con esquemas de tipo proyección.} Actas del XVII CEDYA/
 VII CMA, Universidad de Salamanca (2001).

\bibitem{gr}{\sc  F.~Guill\'en, M.A.~Rodr\'{\i}guez-Bellido. }
{\sl On the strong solutions of the  Primitive Equations in $2D$
  domains. } Nonlinear Analysis: Serie A, Theory and  Methods, {\bf
  50} (5) (2002),
 621-646.



\bibitem{gd}{\sc  F.~Guill\'en-Gonz\'alez,
    D.~Rodr\'{\i}guez-G\'omez. } {\sl Bubble finite elements for the primitive
    equations of the ocean. }
Num. Math., {\bf 101} (4) (2005),  689-728.

\bibitem{G-Re-cras}{\sc F.~Guill\'en-Gonz\'alez, M.V.~Redondo-Neble. }
{\sl Sharp error estimates for a fractional-step method applied to
the $3D$ Navier-Stokes equations. }  C.~R.~Acad.~Sci.~Paris, Ser.~I {\bf
345} (2007),  359-362.

\bibitem{g-j}{\sc  F.~Guill\'en-Gonz\'alez, J.V.~Guti\'errez-Santacreu.}
{\sl Conditional stability and convergence of a fully discrete
scheme for $3D$ viscous fluids models with mass diffusion}. 
SIAM J.~Num.~Anal.,  {\bf 46} (5) (2008), 2276-2308.


\bibitem{g-re-NM}{\sc F.~Guill\'en-Gonz\'alez, M.V.~Redondo-Neble. }
{\sl New error estimates for a viscosity-splitting scheme in time
for  the $3D$ Navier-Stokes equations. } IMA J.~Numer.~Anal.~(2011) {\bf 31} (2), 556-579.

 \bibitem{g-re-space} {\sc F.~Guill\'en-Gonz\'alez,
 M.V.~Redondo-Neble.}
 {\sl Spatial error estimates for a finite element
 viscosity-splitting scheme for  the Navier-Stokes equations.} Int.~J.~Numer.~Anal.~Mod. {\bf 10} (4) (2013), 826-844.


 \bibitem{RANNACHER}{\sc J.G.~Heywood, R.~Rannacher. }  {\it Finite element
approximation of the nonstationary Navier-Stokes problem. IV. Error
analysis for second order time discretization. } SIAM
J.~Numer.~Anal., {\bf 27} (1990), 353-384.



\bibitem{ziane} {\sc I. Kukavica, M. Ziane.} {\it On the regularity of
    the primitive equations of the ocean.} Nonlinearity, {\bf 20}
  (2007), 2739-2753.

\bibitem{l}{\sc R.~Lewandowski. }
{\sl Analyse Math\'ematique et Oc\'eanographie. } Masson (1997).

\bibitem{ltw}{\sc J.L.~Lions, R.~Teman, S.~Wang.} {\it New formulations of
the primitives equations of the atmosphere and applications.}
  Nonlinearity, {\bf 5} (1992),  237-288.

\bibitem{ltw2}{\sc J.L.~Lions, R.~Teman, S.~Wang.} {\it On the
equations of the large scale Ocean.} Nonlinearity, {\bf 5} (1992),
1007-1053.




\bibitem{ortegon}{\sc F. Orteg\'on Gallego.} {\it On distributions
    independent of $x_N$ in certain non-cylindrical domains and a de
    Rham lemma with a non-local constraint.} Nonlinear Analysis,  {\bf
    59} (2004), 335-345.



\bibitem{pe}{\sc J.~Pedlosky. } {\sl Geophysical fluid dynamics. }
Springer-Verlag (1987).


\bibitem{simon}{\sc J. Simon.} {\it  Compact sets in  $L^{p}(0,T; B)$},
 Ann. Mat. Pura  Appl., {\bf 146} (1987),  65--97.

\bibitem{teman3} R.~Temam,  Navier-Stokes equations. Theory and Numerical Analysis, 
North-Holland, 1984.





\bibitem{Z}{\sc M.~Ziane. }
{\sl Regularity Results for Stokes Type Systems.} Applicable Analysis,
{\bf 58} (1995),  263-292.

\end{thebibliography}
\end{document}